\documentclass[12pt]{article}

\usepackage[utf8]{inputenc}

\usepackage{amsfonts}
\usepackage{xcolor}
\usepackage{caption}
\usepackage{mathrsfs}
\usepackage[T1]{fontenc}
\usepackage{tikz}
\usepackage{floatrow}
\usepackage{mathtools}

\usepackage{stackrel}
\usepackage{hyperref}
\usepackage{multicol}

\usepackage{amsmath,amsthm,amssymb,setspace,enumitem,epsfig,titlesec,verbatim,color,array,eurosym,multirow,blkarray,lscape,bm,nicefrac,comment,soul,graphicx,MnSymbol,wasysym,subfigure,multirow}
\usepackage[top=2.5cm,left=2.8cm,right=2.8cm,bottom=2.5cm]{geometry}
\usepackage{float}

\usepackage{cancel}

\usetikzlibrary{decorations.pathreplacing,angles,quotes}

\newtheorem{theorem}{Theorem}
\numberwithin{theorem}{section}

\newtheorem{lemma}[theorem]{Lemma}
\newtheorem{proposition}[theorem]{Proposition}

\newtheorem{definition}[theorem]{Definition}
\newtheorem{remark}[theorem]{Remark}
\definecolor{G1}{rgb}{0.0, 0.5, 0.0}
\definecolor{FG}{rgb}{0.0, 0.5, 0.0}

\allowdisplaybreaks

\title{ An augmented phase plane approach for discrete planar maps: Introducing next-iterate  operators}

\author{
  Streipert, Sabrina H.\\
  \texttt{streipes@mcmaster.ca}
  \and
  Wolkowicz, Gail S.\ K. \\
  \texttt{wolkowic@mcmaster.ca}
}
\date{\today}

\begin{document}

\maketitle

\begin{abstract}
The next-iterate operators and corresponding next-iterate {\it root-sets} and {\it root-curves}  associated with the nullclines   of a planar discrete map are introduced. How to  augment standard phase portraits  that include the nullclines and the direction field, 
 by including the signs of the root-operators  associated with their  nullclines, thus producing an {\it augmented phase portrait},  is described. 
The sign of a next-iterate operator associated with a nullcline determines whether a point is mapped above or below the corresponding nullcline and can,  for example,   identify positively invariant regions.
Using a Lotka--Volterra type competition model, we demonstrate  how to  construct the augmented phase portrait. We show that the augmented phase portrait provides an  elementary, alternative approach  for determining the complete global dynamics of this model. We further explore the limitations and potential  of the augmented phase portrait by considering a Ricker competition model,  a model involving mutualism,  and a predator--prey model.

\end{abstract}

 
 \medskip
 


\section{Introduction}
Phase plane analysis of planar systems of ordinary differential equations with vector fields defined by  continuously differentiable functions has proven very useful for determining
both local and global  dynamics. 
We refer to the phase portrait that includes only the nullclines and the direction and bounds on the slope of the orbits in each of the regions bounded by the nullclines,  as ``standard phase portrait''.   For planar differential equations, it can be used to identify invariant and positively invariant regions. This is because,   by the Poincar\'{e}-Bendixson Theorem (see e.g., \cite{Allen2007, Edelstein1988}), distinct orbits in phase-space cannot intersect and by the continuity of  orbits in phase-space, they  can only cross nullclines in the direction indicated by the direction field.    


A well-known example of the successful application of   phase plane analysis  in the context of  planar systems of ordinary differential equations 
that has been extensively studied (see e.g., \cite{brauer2011, Braun1979,  Edelstein1988, Gause1934, Pielou1969}), 
is the classical two-species competition model:
\begin{equation}\label{Compete_cont}
    x'=r_1 x\left( 1 - \frac{x}{K_1}\right) - \alpha_1 xy, \qquad \qquad \qquad
    y'=r_2 y\left( 1 - \frac{y}{K_2}\right) - \alpha_2 xy,
\end{equation}
where $r_1,r_2>0$  denote  the growth rate, $K_1,K_2>0$  the carrying capacities, and $\alpha_1,\alpha_2>0$ the inter-specific competition impact rates, of species $x$ and $y$, respectively. System \eqref{Compete_cont} was proposed  by Lotka \cite{Lotka}   and Volterra \cite{Volterra1926}.  It is assumed that each species grows logistically in the absence of the other and  both inter- and intra- specific competition reduces each species numbers. 
It is possible  to determine the invariant and positively invariant regions and hence the local and global stability of the equilibria from the standard phase portraits shown in  Fig.~\ref{fig:Edelstein}\footnote{ All of the figures  were produced using Matlab \cite{Matlab:2020b}.}  (see e.g.,  \cite{ Allen2007, Edelstein1988}). For positive initial conditions, in a) and b) there is competitive exclusion  (in a) $y$ excludes $x$, in b) $x$ excludes $y$),  in c) outcomes are  initial condition dependent, and in d) all solutions  converge to the coexistence equilibrium.

\begin{figure}[!ht]
    \centering
 \includegraphics[width=0.45\textwidth, height = 6cm]{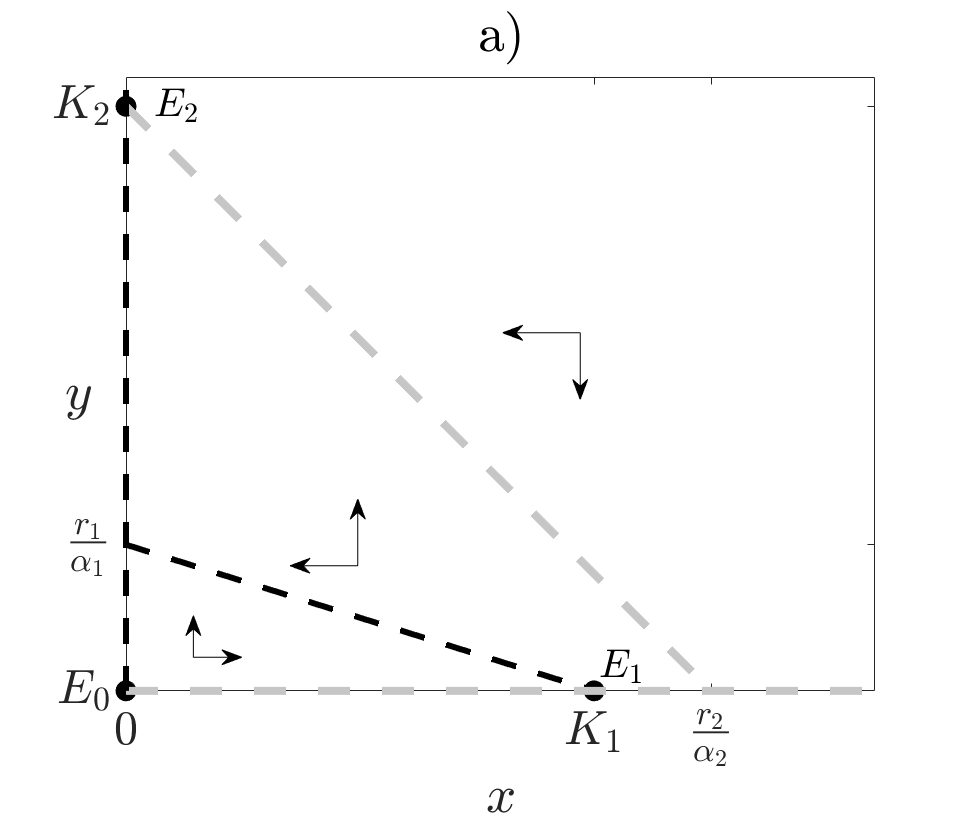}
    \includegraphics[width=0.45\textwidth, height = 6cm]{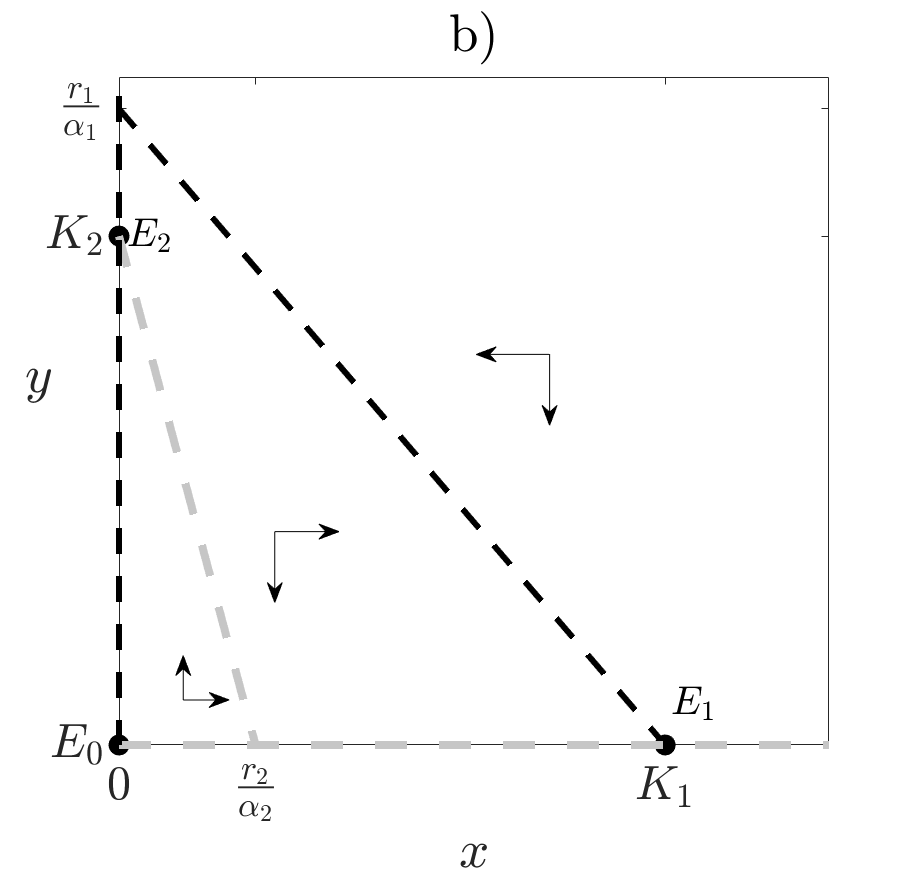}
    
    \includegraphics[width=0.45\textwidth, height = 6cm]{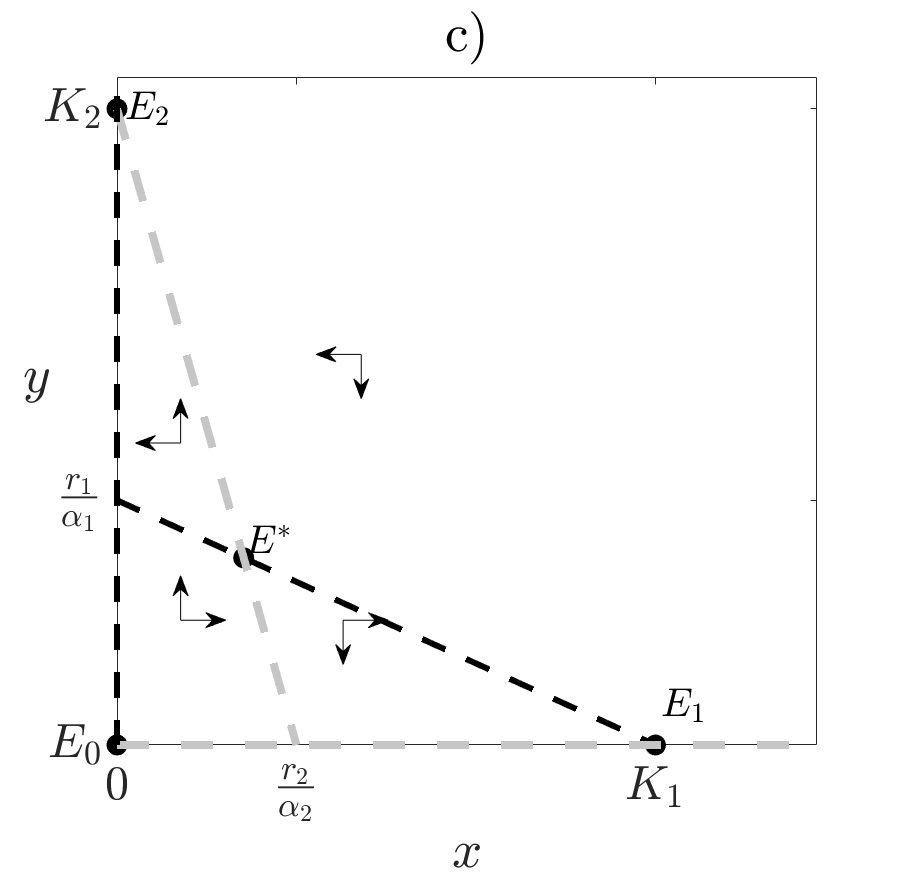}
    \includegraphics[width=0.45\textwidth, height = 6cm]{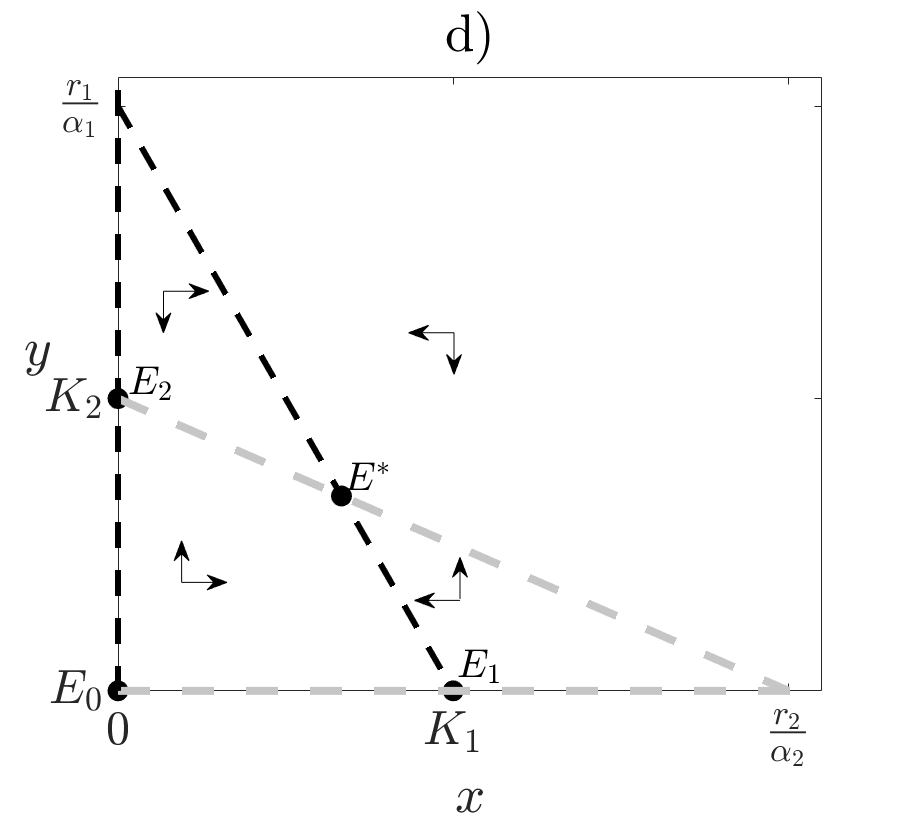}
    \caption{
    Phase portraits for the system of differential equations \eqref{Compete_cont}.
     The   black  dashed lines are the nullclines for population $x$
     and the light gray dashed lines are the nullclines for population $y$.
     The points $E_0, E_{1}$, and  $E_{2}$ are the boundary equilibria and $E^*$ is the coexistence equilibrium.    }
    \label{fig:Edelstein}
\end{figure}

The standard phase  portrait  has   not been as helpful for analyzing planar discrete maps. Unlike for smooth systems of planar ordinary differential equations for which the standard phase portrait can be used to find all of the  invariant and positively invariant regions, for planar discrete maps it is possible for orbits to   jump across one or more nullclines in a single iteration.  Therefore, the standard phase portrait  cannot be  used successfully to detect invariant or positively invariant regions.
This is demonstrated in Fig.~\ref{Fig:Jump}, where the standard phase plane is shown  for the discrete  Ricker competition map:  
\begin{equation}\label{eq:jumps}
    X_{t+1} = X_t \rm{e}^{(0.9 -X_t-0.4 Y_t)}, \qquad \qquad 
    Y_{t+1} =  Y_t \rm{e}^{(1.6-0.3 X_t -Y_t)}.
    \end{equation}

Fig.~\ref{Fig:Jump} shows the first few iterations of  orbits of \eqref{eq:jumps} with the initial conditions indicated by stars. In these phase portrait,  as well as in all of the phase portraits, nullclines will be included using dashed curves, with black curves used for  the $X$-equation and gray curves for  the $Y$-equation. The line segments with arrows indicate the direction and bounds on the slope of the orbits in each of the regions bounded by the nullclines, and will be referred to  simply as the direction field, for convenience. 

In Fig.~\ref{Fig:Jump}a), the orbit jumps across both nullclines and in b) the orbit jumps outside of a region that would be positively invariant if  the  phase portrait were  for a continuous system.  This illustrates   the main  drawbacks with regard to using standard phase portraits to analyze  discrete planar models. Such issues even occur in linear planar maps as pointed out in \cite[p.~48]{Galor2007}. 

\begin{figure}[!htb]
    \centering
    a) \hspace{5cm} b) \hspace{5cm}
    \includegraphics[ scale=0.3]{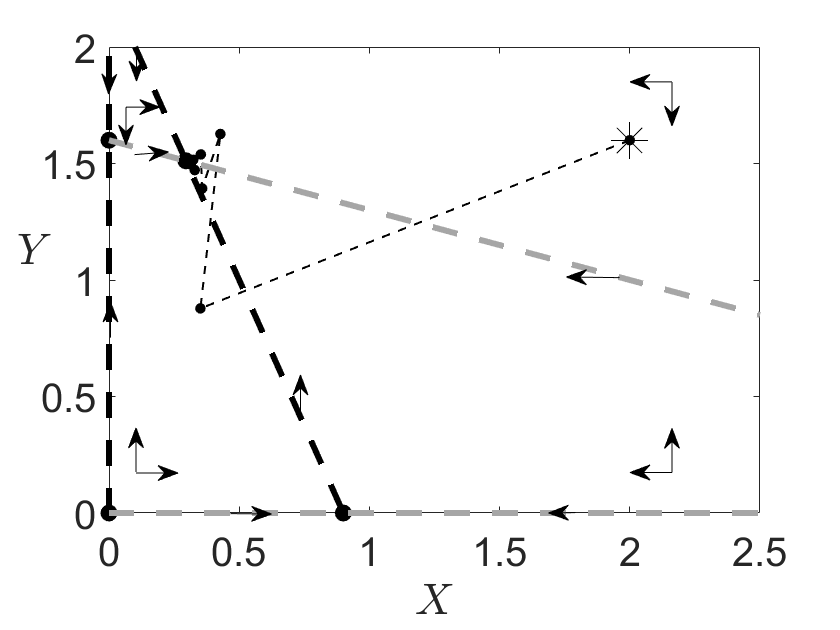}
\includegraphics[ scale=0.3]{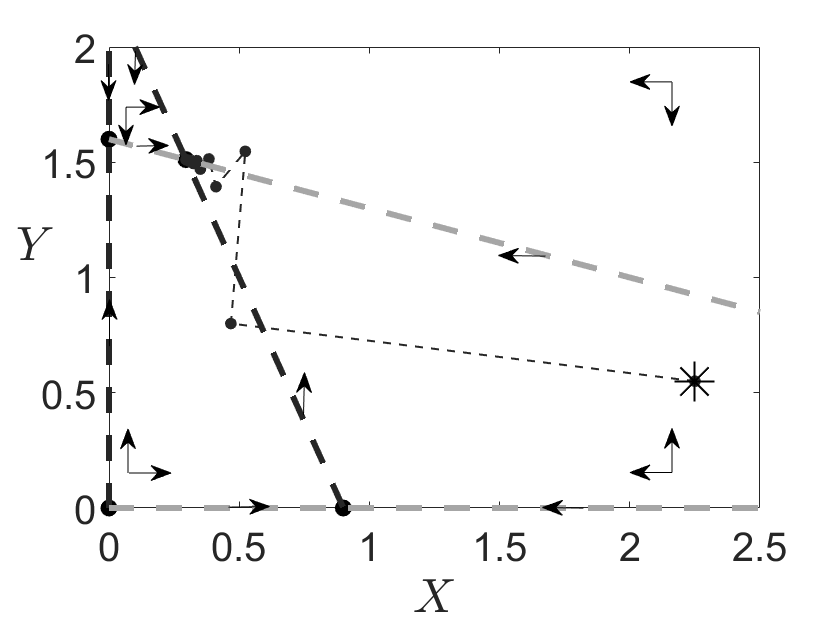}
    \caption{
The standard phase portrait for discrete planar system \eqref{eq:jumps},  including different positive semi-orbits  with the initial point of each indicated by a star.  The nullclines are shown as   dashed curves using  black and gray for the ones   related to the 
     $X$ and $Y$ -equations, respectively. Unlike for systems of ordinary differential equations, orbits of maps can jump over both nullclines as in a). The orbit in b) leaves the  region bounded by the nontrivial nullclines and the $X$-axis in one iteration, a region that would be positively invariant if the phase portrait was for a planar system of ordinary differential equations. For example, the configuration of the nullclines and the direction field is the same as in Fig~\ref{fig:Edelstein} d). }
    \label{Fig:Jump}
\end{figure}
%

To overcome some of these  drawbacks, in Section~\ref{sec:operator},  we introduce next-iterate operators associated with nullclines and the corresponding root-sets and root-curves.  The sign of the next-iterate operator associated with its nullcline determines on which side of that nullcline the next iterate lies. Root-sets determine root-curves  that are curves along which  the next-iterate operator equals zero. Root-curves therefore sub-divide the phase plane into regions in which the sign of the operator is constant.

To show how to augment the standard phase plane by including the signs of the next-iterate operators  and then use the augmented phase plane to analyze planar discrete maps, in  Section~\ref{sec:compete},  we illustrate the method on the following discrete version of \eqref{Compete_cont}, 
\begin{equation}\label{Compete}
    X_{t+1}=F(X_t,Y_t)=\frac{1+r_1}{1+\frac{r_1}{K_1}X_t+\alpha_1 Y_t}X_t, \qquad     Y_{t+1}=G(X_t,Y_t)=\frac{1+r_2}{1+\frac{r_2}{K_2}Y_t+\alpha_2 X_t}Y_t,
\end{equation}
with initial conditions $X_0,Y_0\geq 0$.  The model parameters have the same interpretation as in model \eqref{Compete_cont}.




Model   \eqref{Compete} is well-known and was first derived by Leslie \cite{Leslie1958} who described the possible asymptotic outcomes of \eqref{Compete} as the same as for   model \eqref{Compete_cont}.
More recently, \eqref{Compete} was derived  in \cite{Liu2001}, using a Mickens discretization scheme and in \cite{StWo2} by applying a fitness function approach. The local analysis of \eqref{Compete} (see \cite{Liu2001, Pielou1969}), was  extended in \cite{Baigent2016, Cushing2004,Liu2001}  using different techniques. 
For example, in \cite{Baigent2016}, the idea of a carrying simplex was applied while the analysis in \cite{Liu2001} relied on the theory of monotone  dynamical systems.



 Using the augmented phase portrait to analyze model \eqref{Compete},
 we were able to  determine the complete global dynamics using an alternative, more elementary  method, compared to the approaches used in 
 \cite{ Baigent2016, Cushing2004, Liu2001}. In Section~\ref{sec:ext_limits}, we apply the method to several other systems and discuss some limitations.

\section{The next-iterate operator and associated root-set and  root-curves}\label{sec:operator}

Consider the general planar system 
\begin{equation}\label{gensys}
X_{t+1}=F(X_t,Y_t), \qquad \qquad Y_{t+1}=G(X_t,Y_t).
\end{equation}

Let $Y=\ell(X)$ be a nullcline of \eqref{gensys}. We  introduce the next-iterate operator, root-set, and  root-curve associated with   this nullcline to augment the standard phase portrait to make it more useful for the analysis of  \eqref{gensys}.

\begin{definition}\label{def:L_ell}
The next-iterate  operator associated with the nullcline $Y=\ell(X)$ is the function
\begin{equation*}
\mathcal{L}_{\ell}(X,Y):=G(X,Y)-\ell(F(X,Y)).
\end{equation*}
\end{definition}

By Definition \ref{def:L_ell}, it follows that 
$\mathcal{L}_{\ell}(X_t,Y_t)=Y_{t+1}-\ell(X_{t+1})$, 
so that
 \[\begin{array}{llll}
\mathcal{L}_{\ell}(X_t,Y_t)>0 & \iff& \mbox{\textbf{ next iterate lies above }} & Y=\ell(X),\\[1mm]
\mathcal{L}_{\ell}(X_t,Y_t)=0 & \iff & \mbox{\textbf{ next iterate lies on }}& Y=\ell(X),\\[1mm]
\mathcal{L}_{\ell}(X_t,Y_t)<0 & \iff & \mbox{\textbf{next iterate lies below }}&Y=\ell(X).
\end{array}
\]

Since the sign of the next-iterate operator  tells us  on which side of the associated nullcline  the next iterate lies, it is useful to subdivide the phase plane into regions based on the signs of  the next-iterate operators associated with the nontrivial nullclines and  augment the standard phase portrait by including these signs.

  In all of the phase portraits in this manuscript,   all curves  related to the $X$-equation in \eqref{gensys}  will be black and all curves  related to the  $Y$-equation in \eqref{gensys}  will be gray.  Besides the   dashed curves for the nullclines, we include   '+' and '--' symbols  to indicate the sign of the next-iterate operator in various regions using the  matching colors.  When the root-curves are  included, we will use  solid curves in the matching  colors.

\begin{remark}\label{Remark1}
Definition \ref{def:L_ell} requires that the nullcline can be expressed as a function $Y=\ell(X)$. If this is however not the case, but rather, the nullcline can be expressed as  $X=\kappa(Y)$, then the corresponding next-iterate operator would be defined as 
$$ \widehat{\mathcal{L}}_{\kappa}(X,Y)=F(X,Y)-\kappa(G(X,Y)).$$
In this case, the next-iterate operator identifies next iterates of an orbit to be on the ``left'' or the ``right'' of the nullcline 
$X=\kappa(Y)$ instead of ``above'' or ``below''. In this case, the  following Definitions~\ref{def:Sell} and \ref{def:rell}   of  root-set and root-curves  would have to be adjusted accordingly. 
\end{remark}

\begin{definition}\label{def:Sell}
The next-iterate  root-set  (in short: root-set) associated with the nullcline $Y=\ell(X)$ is the set
\begin{equation*}
S_{\ell}:=\{(X,Y)\in \mathbb{R}^2\, \colon\,  \mathcal{L}_{\ell}(X,Y)=0\}. 
\end{equation*}
\end{definition}


\begin{definition}\label{def:rell}
 The next-iterate root-curves (in short: root-curves)  associated with the nullcline $Y=\ell(X)$ are  curves  
   $Y=r(X)$ or $X=R(Y)$ that satisfy $\mathcal{L}_{\ell}(X,r(X))=0$ or $\mathcal{L}_{\ell}(R(Y),Y)=0$.
\end{definition}

\begin{lemma}\label{Lemnoint}
Let $S_\ell$ be the root-set defined  in Definition~\ref{def:Sell} associated with the nullcline $Y=\ell(X)$ of \eqref{gensys}.  Let $\mathcal{E}_{\ell}$ denote the subset of  
 equilibria of \eqref{gensys} that  lie on  $Y=\ell(X)$.

\begin{enumerate}
\item[a)] If $Y=\ell(X)$ is a nullcline  for the $X$-equation, that is $F(X,\ell(X))=X$, \\then
$S_\ell\cap \{(X,Y): Y=\ell(X)\}  = \mathcal{E}_{\ell}.$ 
\item[b)] If $Y=\ell(X)$ is a nullcline for the $Y$-equation, that is $G(X,\ell(X))=\ell(X)$, and $Y=\ell(X)$ is injective, then
$S_\ell\cap \{(X,Y): Y=\ell(X)\}  = \mathcal{E}_{\ell}.$
\end{enumerate}
\end{lemma}

\begin{proof} 
Assume that $(X,Y)\in \mathcal{E}_{\ell}.$ Then, $Y=\ell(X)$  and
$\mathcal{L}_{\ell}(X,Y)=G(X,Y)-\ell(F(X,Y))=Y-\ell(X)=0.$ Therefore, $\mathcal{E}_{\ell}\subseteq \mathcal{S}_{\ell}\cap \{(X,Y) : \, Y=\ell(X)\}.$

{\it a)} Assume  $(X,Y)\in S_\ell\cap \{(X,Y): Y=\ell(X)\}$,
where  $Y=\ell(X)$ is a  nullcline for the $X$-equation so that  $F(X,\ell(X))=X$.
Since $(X,Y)\in S_{\ell}$, $\mathcal{L}_{\ell}(X,Y)=0$ and therefore, 
$G(X,Y)=\ell(F(X,Y))$. Since $Y=\ell(X)$, 
$G(X,Y)=G(X,\ell(X))=\ell(F(X,\ell(X))=\ell(X)=Y$.  Thus, 
$X=F(X,Y)$ and $Y=G(X,Y)$, and therefore  $(X,Y)\in \mathcal{E}_{\ell}$, completing the proof for a).

{\it b)} Assume that $Y=\ell(X)$ is injective and is a nullcline for the $Y$-equation so that $Y=G(X,Y)$. If $(X,Y)\in S_{\ell}\cap \{(X,Y)\colon Y=\ell(X)\}$, then $0=\mathcal{L}_\ell(X,Y)=G(X,Y)-\ell(F(X,Y))$,   so that
$\ell(X)=Y=G(X,Y)=\ell(F(X,Y))$. Since  $Y=\ell(X)$ is injective, $F(X,Y)=X$.  Therefore, $(X,Y)\in \mathcal{E}_{\ell}$, completing the proof for b).
\end{proof}

\begin{remark}
If $Y=\ell(X)$ is a nullcline for $Y$ where  $Y=\ell(X)$ is not injective, and if instead, the nullcline can be expressed as a function $X=\kappa(Y)$, then the result in Lemma \ref{Lemnoint} still holds  for the next-iterate operator 
$\widehat{\mathcal{L}}_{\kappa}(X,Y)$ defined in Remark \ref{Remark1}, with Definitions \ref{def:Sell} and  \ref{def:rell}  adjusted accordingly.
\end{remark}

\begin{remark}\label{intersectgen}

Let $Y=\ell_1(X)$ and $Y=\ell_2(X)$ be the nullclines associated with the $X${\color{red}-} and $Y${\color{red}-}equations, respectively, i.e., $F(X,\ell_1(X))=X$ and $G(X,\ell_2(X))=\ell_2(X)$. Then, $(X,Y)\in S_{\ell_1}\cap S_{\ell_2}$,   where  $S_{\ell_1}$ and $S_{\ell_2}$ are the corresponding root-sets, if and only if $(X,Y)$ is an equilibrium or is  mapped in one iteration to an equilibrium.  
\end{remark}

\section{Analysis of (\ref{Compete}) using the Augmented Phase Portrait}
\label{sec:compete}

For \eqref{Compete}, we define the {\it competitive efficiency} of species $X_i$ with competitor $X_j$ as
\begin{equation}\label{def:Cij}
\mathcal{C}_{ij}:=\frac{r_i}{\alpha_i}-K_j, \qquad \qquad i\neq j; \, \, i,j\in \{1,2\}. \end{equation}
 The relative values of these competitive efficiencies will be shown to determine the asymptotic outcome of the solutions. 
 



 Model \eqref{Compete}  satisfies the {\it Axiom of Parenthood} \cite{Edelstein1988, Hutchinson1978}, that is, every new generation must have had a parent generation so that if $X_0=0$, then $X_t=0$ for all $t\geq 0$. Similarly,  if $Y_0=0$, then  $Y_{t}=0$, for all $t\geq 0$. 
Therefore, each axis bounding the first quadrant is  invariant.
For all $t\geq 0$, if  $X_0>0$, then $X_t>0$ and if $Y_0>0$, then $Y_t>0$.  Therefore, solutions with positive initial conditions cannot become negative. 

In this section, we use model \eqref{Compete} with initial conditions $X_0,Y_0\geq 0$ to illustrate how to construct the augmented phase portrait  and then use it to determine the global dynamics of \eqref{Compete}. 

\subsection{Construction of the Augmented Phase Portrait  for (\ref{Compete})}

\subsubsection{Step I: Nullclines, Equilibria,  and Direction Field}


First, we obtain the nullclines and determine the direction of  component-wise monotonicity in each of the regions separated by the nullclines.

The  nullclines for competitor population $X$ are the vertical line $X=0$ and 
\begin{equation}\label{eq:Linex}
    Y=h(X)=\frac{r_1}{\alpha_1 K_1}(K_1-X),
\end{equation} 
so that for $X>0$  and $Y>0$,
\begin{equation}\label{SgnDeltax}
     F(X,Y)-X  = \frac{\alpha_1 X}{1+\frac{r_1}{K_1}X+\alpha_1 Y}\left(h(X)-Y\right) = \quad \begin{cases} \quad <\quad 0, & \mbox{ if}\quad Y>h(X),\\
   \quad  =\quad 0, & \mbox{ if}\quad Y=h(X),\\
   \quad  >\quad 0, & \mbox{ if}\quad  Y<h(X).\end{cases}
\end{equation}
The  nullclines for competitor population $Y$ are  the  horizontal  line $Y=0$ and  the line
\begin{equation}\label{Liney}
Y=k(X)=\frac{K_2}{r_2}(r_2-\alpha_2 X),
\end{equation}  
so that for $X>0$ and $Y>0$,   
\begin{equation}\label{SgnDeltay}
G(X,Y)-Y  = \frac{\frac{r_2}{K_2}Y}{1+\frac{r_2}{K_2} Y +\alpha_2 X}\left(k(X)-Y\right)  =\quad \begin{cases} \quad <\quad 0, & \mbox{ if}\quad  Y>k(X)\\
\quad =\quad 0, & \mbox{ if}\quad  Y=k(X)\\
\quad >\quad 0, & \mbox{ if}\quad  Y<k(X).\end{cases}
\end{equation}

The set of biologically relevant equilibria of \eqref{Compete},  denoted $\mathcal{E}$,  always contains three boundary equilibria:
\begin{equation*}
    E_0=(0,0), \qquad  \quad E_{1}=(K_1,0), \qquad \quad E_{2}=(0,K_2).
\end{equation*}
 When $C_{12} \cdot C_{21}>0$,  the two nullclines $Y=h(X)$ and $Y=k(X)$  cross in the interior of the first quadrant  at  a unique  coexistence    equilibrium, $E^*$, also contained in $\mathcal{E}$, where
\begin{equation}\label{coexequ}
E^*=(X^*,Y^*)=\left(\frac{r_2 K_1(\alpha_1 K_2- r_1)}{(\alpha_1 \alpha_2 K_1 K_2 - r_1 r_2)},\frac{r_1 K_2 (\alpha_2 K_1 - r_2)}{(\alpha_1 \alpha_2 K_1 K_2 - r_1 r_2)}\right) \in (0,K_1)\times (0,K_2).
\end{equation}
In the special case when  $C_{12}=C_{21}=0$,  and so $h(X)= k(X)$ for all $X$,  the entire line segment of equilibrium points, 
\begin{equation*}
    \mathcal{E}_X^*=\{(X,Y)\, : \, Y=h(X), \, 0\leq X\leq K_1\},
\end{equation*}
 is contained in $\mathcal{E}$. 
Note that in this case, $E_{1}$ and $E_{2}$ are in  $\mathcal{E}_X^*$, and  $\mathcal{E}=\{E_{0}\} \cup  \mathcal{E}_X^*$.

It follows that the standard phase portrait for the discrete map \eqref{Compete} looks the same as the phase portrait for the system of differential equations \eqref{Compete_cont} shown in Fig.~\ref{fig:Edelstein}  if the labels $x$ and $y$ on the axes are replaced by $X$ and $Y$, respectively.


\subsubsection{Step II: Next-iterate operators, Root-Sets, and Root-Curves }

By Definition~\ref{def:L_ell}, the {\it next-iterate operators} associated with the positive nullclines for  competitors $X$ and $Y$ are given by  
\begin{equation*}
\mathcal{L}_h(X,Y):=G(X,Y)-h(F(X,Y)) \qquad \mbox{ and }\qquad
\mathcal{L}_k(X,Y):=G(X,Y)-k(F(X,Y)),
\end{equation*}
 respectively. Then, 
 \begin{samepage}
 \begin{align}
        \mathcal{L}_h(X,0)&=G(X,0)-h(F(X,0))=0-h(F(X,0))=-h(F(X,0)),   \label{LhX0}\\             
        \mathcal{L}_k(X,0)&=G(X,0)-k(F(X,0))=0-k(F(X,0))=-k(F(X,0)),\label{LkX0}\\
        \mathcal{L}_h(0,Y)&=G(0,Y)-h(F(0,Y))=G(0,Y)-h(0)=G(0,Y)-\frac{r_1}{\alpha_1},  
        \label{LhY0}\\     
        \mathcal{L}_k(0,Y)& =G(0,Y)-k(F(0,Y))=G(0,Y)-k(0)=G(0,Y)-K_2. \label{LkY0}        
\end{align}
\end{samepage}
 By \eqref{LhX0}, $\mathcal{L}_h(X,0)<0$ for all $X\in [0,K_1)$ and, by \eqref{LkY0}, $\mathcal{L}_k(0,Y)<0$ for all $Y\in [0,K_2)$.

\begin{lemma}\label{usefulhk}
Assume that $X>0$ and $Y>0$. 

\begin{enumerate}
    \item[a)] If $h(X)<k(X)$, then $\mathcal{L}_h(X,h(X))>0$ for $X\in (0,K_1)$ and $\mathcal{L}_k(X,k(X))<0$ for $X\in \left(0, \frac{r_2}{\alpha_2}\right)$. 
    \item[b)] If $k(X)<h(X)$, then $\mathcal{L}_h(X,h(X))<0$ for  $X\in (0,K_1)$ and  $\mathcal{L}_k(X,k(X))>0$ for $X\in \left(0,\frac{r_2}{\alpha_2}\right)$ 
\end{enumerate}
\end{lemma}

\begin{proof} 
Assume that $X>0$ and $Y>0$.   We only prove {\it a)}, since the argument for  {\it b)}   is similar. Assume therefore that $h(X)<k(X)$ and first that  $X\in(0,K_1)$. Then, $0<h(X)$. By Definition~\ref{def:L_ell},
 $\mathcal{L}_{h}(X,h(X))= G(X,h(X))-h(F(X,h(X)))=G(X,h(X))-h(X)$.
Since  $0<h(X)<k(X)$,    by \eqref{SgnDeltay},  $G(X,h(X))-h(X))>0$, i.e., $\mathcal{L}_{h}(X,h(X))>0$.

Next, assume that  $X\in(0,\frac{r_2}{\alpha_2})$.  Then $k(X)>\max\{0,h(X)\}\geq 0$.
Thus, by \eqref{SgnDeltax}, $F(X,k(X))<X$. 
Since $k(X)$ is a decreasing function of $X$,
$k(X)<k(F(X,k(X))$ and therefore
  $\mathcal{L}_{k}(X,k(X))= G(X,k(X))-k(F(X,k(X)))  =k(X)-k(F(X,k(X)))<0$.  Hence, both results in {\it a)} follow.
\end{proof}


For \eqref{Compete}, the next-iterate operators are of the form  
\begin{equation}\label{calcLs}
\begin{split}
    \mathcal{L}_h(X,Y)&=\frac{ N_h(X,Y)}{\alpha_1 (K_1 + r_1 X + \alpha_1 K_1 Y) (K_2 + \alpha_2 K_2 X + r_2 Y)},\\
    \mathcal{L}_k(X,Y)&=\frac{N_k(X,Y)}{r_2 (K_1 + r_1 X + \alpha_1 K_1 Y) (K_2 + \alpha_2 K_2 X + r_2 Y)},
\end{split}
\end{equation}
where $N_h(X,Y)$ and $N_k(X,Y)$ are quadratic polynomials in $X$ and $Y$. The precise expressions are provided in  Appendix \ref{PfcalcLs} with the expressions for the root-sets and root-curves.

 The proof of the following Lemma is based on the fact that if a point $(X,Y)$ is in both root-sets, that is, $(X,Y)\in S_k\cap S_h$, then this point is mapped  directly to an equilibrium. The details are provided in Appendix \ref{PfrootintersectD1}.

\begin{lemma}\label{rootintersectD1}
 If $C_{12}\cdot C_{21}>0$,
 then 
 $$S_k\cap S_h \cap \{(X,Y)\colon \, 0<X\leq X^*,\, \,  0<Y\leq Y^*\}=E^*$$
  and 
$$S_k\cap S_h \cap \{(X,Y)\colon \,  X\geq X^*, \, \, Y\geq Y^*\}=E^*,$$
where $E^*=(X^*,Y^*)$ is the coexistence equilibrium given in \eqref{coexequ}.
\end{lemma}

Lemma \ref{rootintersectD1} implies that root-curves associated with the nullclines $Y=h(X)$ and $Y=k(X)$ cannot intersect in 
$\{(X,Y)\colon \, 0<X\leq X^*, \, \, 0<Y\leq Y^*\}\backslash E^*$
 or in $\{(X,Y)\colon \, X\geq X^*, \, \, Y\geq Y^*\}\backslash E^*$.

\subsection{Global Analysis of (\ref{Compete}) using the Augmented Phase Portrait}

In this section, we illustrate how to use the augmented phase portrait to obtain the global dynamics of \eqref{Compete} based on the signs of the competitive efficiencies defined in \eqref{def:Cij}. However, first we provide some preliminary results.  

\begin{theorem}\label{thm:E0}
Consider \eqref{Compete} with $X_0,Y_0\geq 0$  and $X_0Y_0=0$.
\begin{itemize}
    \item[a)] If $X_0=Y_0=0$, then  $(X_t,Y_t)=E_0$, for all $t\geq 0$.     \item[b)] If $X_0=0$ and $Y_0>0$, then $\lim_{t\to \infty}(X_t,Y_t)=E_{2}$.
    \item[c)] If $X_0>0$ and $Y_0=0$, then $\lim_{t\to \infty}(X_t,Y_t)=E_{1}$.
\end{itemize}
\end{theorem}
The proof  is omitted, since it follows immediately from the structure of \eqref{Compete} and  the well-known results for the Beverton-Holt model (see \cite[Section 3.2]{Allen2007}).  This  theorem could  also  be proved using the augmented phase portrait approach, since the root-curves associated with each trivial nullcline coincides with its nullcline. This implies that the trivial nullclines, i.e., the $X$ and $Y$ axes, are invariant and also that the  interior of the first quadrant is invariant.


\subsubsection{Case I: $C_{12}=C_{21}=0$} 

By the definition   of $C_{ij}$ in \eqref{def:Cij}, $\frac{r_1}{\alpha_1}=K_2$ and $\frac{r_2}{\alpha_2}=K_1$, and so $h(X)=k(X)$, for all $X\in \mathbb{R}$.

The standard phase portrait  determined from 
 \eqref{SgnDeltax} and  \eqref{SgnDeltay} is shown in  
Fig.~\ref{genDFC1}a). Two regions  of component-wise monotonicity in $\mathbb{R}_+^2=(0,\infty)^2$ are identified:
\begin{equation*}
 \mathcal{R}_1=\left\{(X,Y)\in \mathbb{R}_+^2\colon Y< h(X)=k(X)\right\} \quad 
 \mathcal{R}_2=\left\{(X,Y)\in \mathbb{R}_+^2 \colon h(X)=k(X)<Y\right\}.
\end{equation*}

\begin{figure}[!ht]
    \centering
    (a) \hspace{72mm} (b)
     \includegraphics[scale=.35]{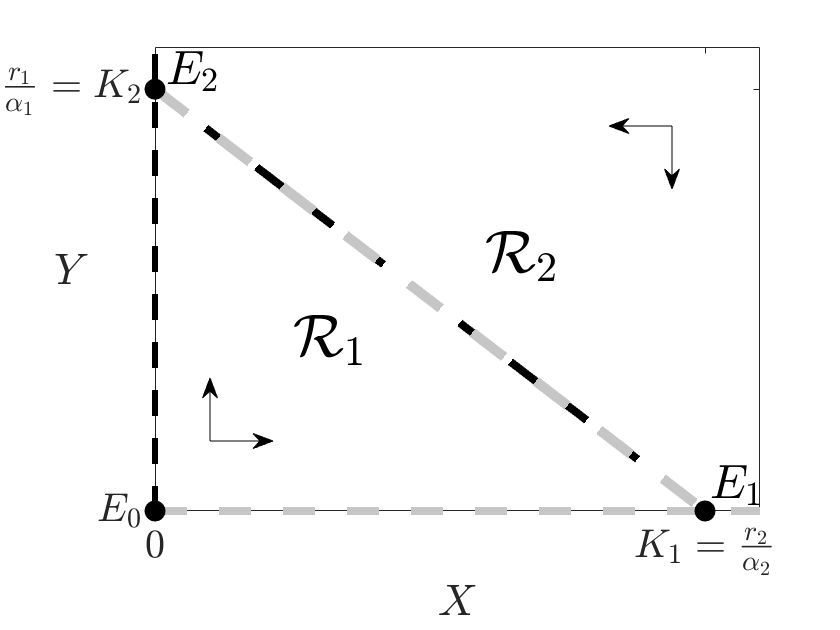}
     \includegraphics[scale=.35]{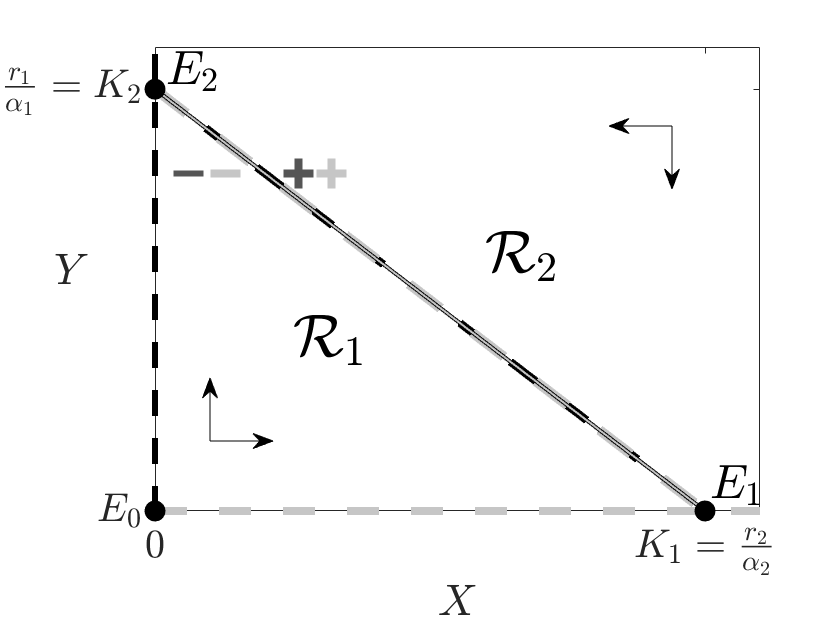}
    \caption{a) The  standard phase portrait in the case when $C_{12}=C_{21}=0$. Since $h(X)=k(X)$, the nontrivial $X$ and $Y$ nullclines coincide and result in the line of equilibria,  $\mathcal{E}^*_X$. b) The graph in a)  augmented by including  the root-curves associated with the  nontrivial nullclines for competitors $X$ and $Y$,  that overlap their nullclines in this case, and the signs of the next-iterate operators in the regions separated by the root-curves. Each root-curve  associated with its nontrivial nullcline is identical to its nontrivial nullcline. Since the next-iterate operators are both negative in $\mathcal{R}_1$, any point in $\mathcal{R}_1$ is mapped below the competitor $X$ and competitor $Y$ nullclines and therefore remains in $\mathcal{R}_1$. Thus, $\mathcal{R}_1$ is  invariant. Similarly, the positive signs of both next-iterate root-operators in $\mathcal{R}_2$ imply that a point in $\mathcal{R}_2$ is mapped to a point above both nontrivial nullclines and therefore remains in $\mathcal{R}_2$. Thus, $\mathcal{R}_2$ is also invariant.}
    \label{genDFC1}
\end{figure}

\noindent In this case, the standard phase portrait alone cannot be used to prove the stability of the equilibria in 
$\mathcal{E}_X^*$. We need
to use the augmented phase portrait that includes the signs of the next-iterate operators associated with the nullclines to first prove that orbits cannot jump back and forth across the nullclines.


The proof of the following result is due to \eqref{calcLs} (see the details in Appendix \ref{PfLhLkCase1}. 

\begin{lemma}\label{LhLkCase1}
Assume $C_{12}=C_{21}=0$. 
\begin{align}
    \mathcal{L}_h(X,Y)&\quad \begin{cases}
 \quad    <\quad 0,\quad & \quad \mbox{if }  \,(X,Y)\in \mathcal{R}_1, \\
 \quad    >\quad 0, \quad &  \quad \mbox{if }  \,(X,Y)\in \mathcal{R}_2 .\end{cases} \label{signC1hb} \\
\mathcal{L}_k(X,Y)&\quad \begin{cases} \quad < \quad 0,\quad & \, \mbox{if} \quad (X,Y)\in \mathcal{R}_1,\\
\quad >\quad 0, \quad &\, \mbox{if} \quad (X,Y)\in \mathcal{R}_2.\end{cases}\label{signC1kb}
\end{align}
\end{lemma}

Including the sign of the next-iterate operator in  the regions separated by the root-curves, we obtain  the augmented phase portrait shown in  Fig.~\ref{genDFC1}b), from which it follows immediately that orbits cannot jump between regions  $\mathcal{R}_1$ and $\mathcal{R}_2$, i.e., each of these regions is invariant.

\begin{theorem}\label{Thm:hequalk}
If  $C_{12}=C_{21}=0$, then  every equilibrium point $(\widehat{X},\widehat{Y})\in \mathcal{E}_X^*$ is a stable equilibrium and any orbit with $(X_0,Y_0)\neq (0,0)$ converges to a point in $\mathcal{E}_X^*$.
\end{theorem}

\begin{proof}
We use the augmented phase portrait shown in Fig.~\ref{fig:example}.
From the two `--' signs  in $\mathcal{R}_1$ and  two `+'  signs in $\mathcal{R}_2$, obtained from \eqref{signC1hb} and \eqref{signC1kb},  it follows immediately that both $\mathcal{R}_1$ and $\mathcal{R}_2$ are invariant. Select an arbitrary equilibrium point $(\widehat{X},\widehat{Y})\in \mathcal{E}^*_X \backslash (E_1\cup E_2)$ and   any open set, $U$, containing $(\widehat{X},\widehat{Y})$.
    There exists $\epsilon>0$ such that  $U_{\epsilon}=\{(X,Y)\in (0,\infty)^2 \, \colon \, \|(X,Y)-(\widehat{X},\widehat{Y})\|_{\infty} < \epsilon\} \subset U$.  Take the rectangle $V\subseteq U_{\epsilon}$ such that both its upper left corner and its lower right corner are on the nullcline $Y=h(X)$.  
     If $(\widehat{X},\widehat{Y})=E_1$ or $E_2$, select $U$ containing $(\widehat{X},\widehat{Y})$, open relative to $[0,\infty]^2$. If $(\widehat{X},\widehat{Y})=E_1$,
     let  $U_{\epsilon}=\{(X,Y)\in (0,\infty)\times [0,\infty)\, \colon \, \|(X,Y)-(\widehat{X},\widehat{Y})\|_{\infty} < \epsilon\} \subset U$ and take the rectangle $V\subseteq U_\epsilon$ with its upper left corner on the nullcline and if $(\widehat{X},\widehat{Y})=E_2$, let $U_{\epsilon}=\{(X,Y)\in [0,\infty)\times (0,\infty)\, \colon \, \|(X,Y)-(\widehat{X},\widehat{Y})\|_{\infty} < \epsilon\} \subset U$ and take a rectangle $V\subseteq U_\epsilon$ with its lower right corner on the nullcline. (If $K_1=K_2$, take $V=U_{\epsilon}$.) From the directions field, in all cases, the rectangle $V$ is positively invariant.
\end{proof}

\begin{figure}[ht!]
\hspace{9mm} a) \, $K_1>K_2$ \hspace{57mm} b)\,  $K_1<K_2$
\includegraphics[scale=0.45,trim = 1.2cm 1cm 1cm 0cm, clip]{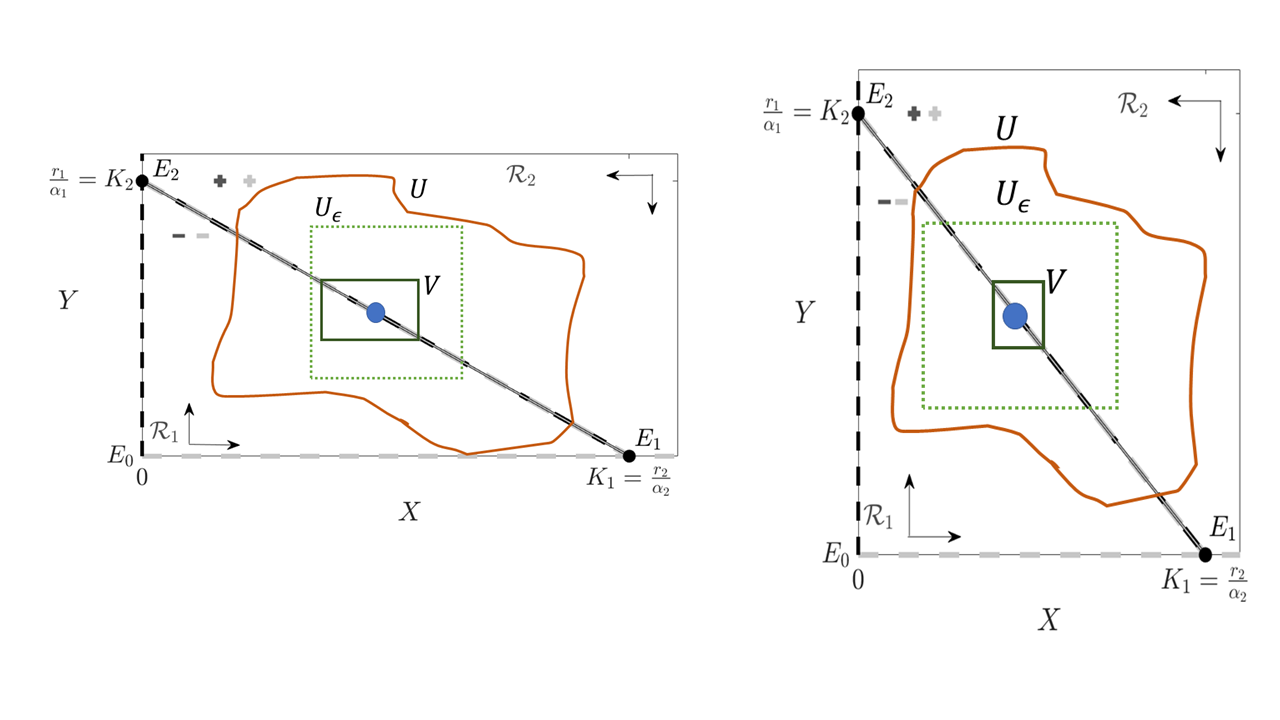}
    \caption{ Graphs illustrating  the use of the augmented phase portrait when $C_{12}=C_{21}=0$,
    to prove the local stability of  every $(\widehat{X}, \widehat{Y})\in \mathcal{E}_X^*$. See the proof of Theorem~\ref{Thm:hequalk}.
    }
    \label{fig:example}%
\end{figure}

\begin{remark}
In the case when $C_{12}=C_{21}=0$,  the eigenvalues of the Jacobian matrix evaluated at any $(X,h(X))\in E^*_X$,  $X\in[0,K_1]$,  are  $\lambda_1(X),\lambda_2(X)$ with
$$0<\lambda_1(X)=\frac{1 + r_1\left(1 - \frac{X}{K_1}\right) +\frac{X}{K_1}  r_2}{(1 + r_1) (1 + r_2)}<1 \qquad \mbox{and}  \qquad \lambda_2(X)=1.$$
 Determining the stability using the classical method of calculating the eigenvalues of the Jacobian matrix is therefore inconclusive. 
 Instead, the next theorem provides a global analysis using the  augmented phase portrait and the definition of a stable equilibrium point and does not require the  calculation of the eigenvalues.  It is the invariance of each of the regions $\mathcal{R}_1$
 and $\mathcal{R}_2$ that is the key ingredient. This can be determined from the augmented phase portrait, but not from the standard phase portrait. 
\end{remark}

\subsubsection{Case II:  $C_{12}C_{21}<0$} 

In this case, competitive efficiencies have opposite signs.  We will show that for all positive initial conditions, there is competitive exclusion, that is, the population with the positive competitive efficiency  wins the competition and drives the other competitor to extinction. 

For competitive efficiencies with opposite signs, the nontrivial nullclines of competitors $X$ and $Y$ do not intersect in the first quadrant and so there  is no coexistence equilibrium. We  only provide an analysis for the case when 
$C_{12}<0$ and $C_{21}>0$, i.e.,    $\frac{r_1}{\alpha_1}< K_2$ and $\frac{r_2}{\alpha_2}>K_1$.  The proofs in the case when  $C_{12}>0$ and $C_{21}<0$ follow by  interchanging the roles of $X$ and $Y$.

As in Case I, the standard phase portrait is not sufficient to determine the global dynamics. 
 For example, additional  information is required to rule out  orbits jumping between regions $\mathcal{R}_1$ and $\mathcal{R}_3$ or jumping out of $\mathcal{R}_2$, where:
\begin{align*}
\mathcal{R}_1&:=\{(X,Y)\in \mathbb{R}_+^2 \colon \, Y< h(X)\},\qquad 
\mathcal{R}_2:=\{(X,Y)\in \mathbb{R}_+^2 \colon \, h(X)\leq Y\leq k(X)\},\\
&\hspace{30mm} \mathcal{R}_3:=\{(X,Y)\in \mathbb{R}_+^2 \colon \,  k(X)<Y\}.
\end{align*}


To  prove this behavior does not occur,  we augment the standard phase   portrait  by including the signs of the next-iterate operators associated with the nontrivial nullclines. We then use the  augmented phase portrait, shown in Fig.~\ref{Fig:C2general}a), in the proof    that $E_2$ is globally asymptotically stable and $E_0$ and $E_1$ are unstable (see \ref{C2Thm1}).  But first we need some preliminary results.

 The proof of the next lemma  relies on \eqref{calcLs} and is provided in Appendix \ref{PfLemsignCase2}. 

\begin{samepage}
\begin{lemma}\label{LemsignCase2} Assume $C_{12}<0$ and $C_{21}>0$.
\begin{enumerate}
    \item[a)] $\mathcal{L}_h(X,Y) >0 $, for all $(X,Y)\in \mathcal{R}_2\cup \mathcal{R}_3$.
\item[b)]  $\mathcal{L}_k(X,Y) <0 $, for all $(X,Y)\in \mathcal{R}_1\cup \mathcal{R}_2$.
\end{enumerate}

\end{lemma}
\end{samepage}

\begin{figure}[ht!]
    \centering
    (a) \hspace{65mm} (b)
    
 \includegraphics[scale=.3]{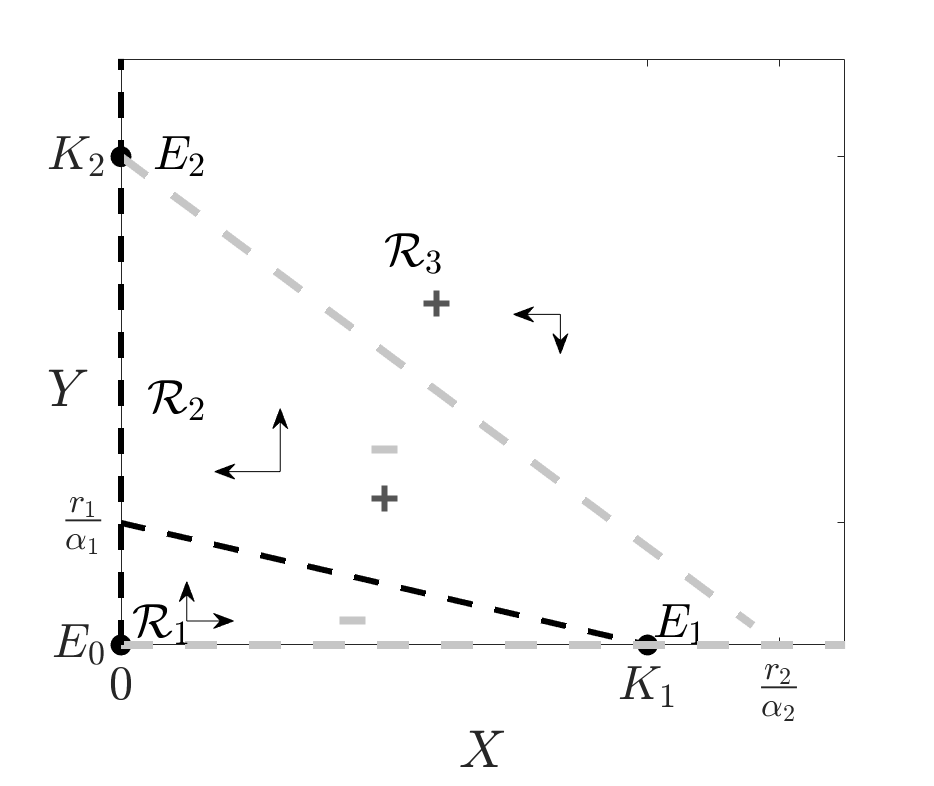}
 \includegraphics[height= 6.25cm, width=7.25cm]{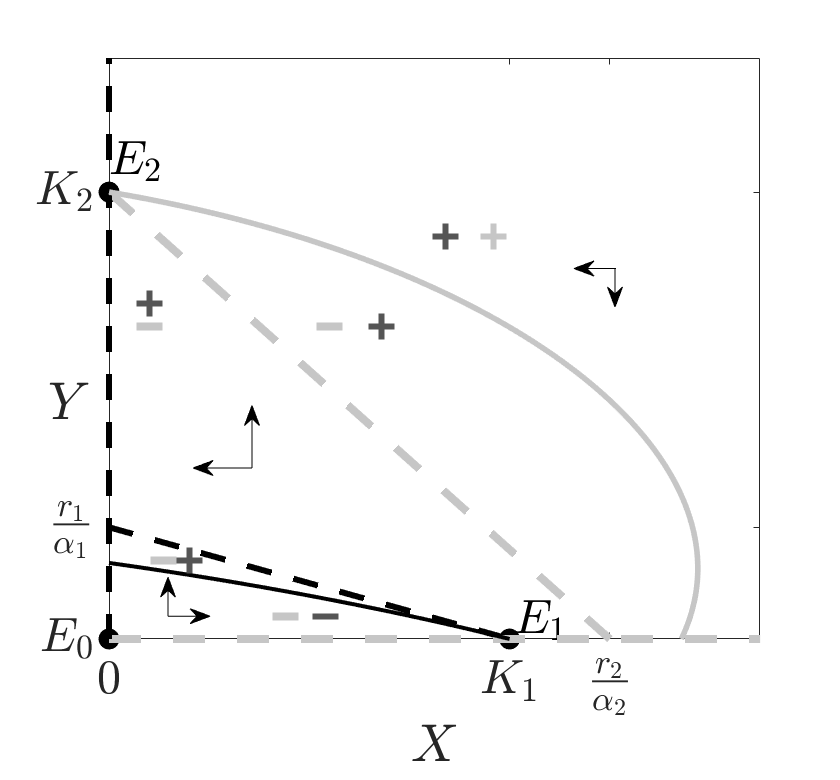}
          \caption{ Augmented   phase portraits in the case when $C_{12}<0$ and $C_{21}>0$. The standard phase portrait is augmented in a)
          by  including only the
         signs (based on the results in Lemma~\ref{LemsignCase2})  of the next-iterate operators associated with the nullclines $Y=h(X)$ and $Y=k(X)$, needed to prove the global asymptotic stability of $E_2$. This phase portrait is prototypical.     An example including the root-curves is shown in b) for parameter values:  $\alpha_1=\alpha_2=1$, $r_1=\frac{1}{2}$, $r_2=\frac{5}{8}$, $K_1=\frac{1}{2}$, $K_2=2$. Once the root-curves are included, the signs of the next-iterate operators in all regions can be included.
        }
         \label{Fig:C2general}
\end{figure}

\noindent Next we show how to use the augmented phase portrait  in Fig.~\ref{Fig:C2general}a),
 to prove that region $\mathcal{R}_2$ is positively invariant.

\begin{proposition}\label{C2invariant}
 If $C_{12}<0$ and $C_{21}>0$, then   $\mathcal{R}_2$ is positively invariant.
\end{proposition}

\begin{proof}
We use the augmented phase portrait shown in Fig.~\ref{Fig:C2general}a). Let $(X_t,Y_t)\in \mathcal{R}_2$. Based on the black `+' symbol in $\mathcal{R}_2$,  obtained from  Lemma~\ref{LemsignCase2}{\it a)}, $\mathcal{L}_h(X_t,Y_t)>0$ so that $(X_{t+1},Y_{t+1})$ remains  above the  nullcline $Y=h(X)$. The gray `--' symbol in $\mathcal{R}_2$, obtained from Lemma~\ref{LemsignCase2}{\it b)}, indicates that $\mathcal{L}_k(X_t,Y_t)<0$. Thus,  $(X_{t+1},Y_{t+1})$ remains below  the nullcline $Y=k(X)$. Thus, $(X_{t+1},Y_{t+1})\in \mathcal{R}_2$.
\end{proof}

\begin{theorem}\label{C2Thm1}
If $C_{12}<0$ and $C_{21}>0$, then any orbit with initial condition $X_0,Y_0>0$ converges to $E_{2}$.  Furthermore, $E_0$ is a repeller and $E_1$ s a saddle.
\end{theorem}

\begin{proof}
Proposition~\ref{C2invariant} and the direction field in $\mathcal{R}_2$ imply that orbits entering $\mathcal{R}_2$ converge to $E_2$. 
To show that $E_2$ is globally asymptotically stable with respect to all solutions with positive initial conditions, it  suffices to show that all such orbits either converge to $E_2$ or eventually enter $\mathcal{R}_2$. We do so using the augmented phase portrait in Fig.~\ref{Fig:C2general}a).

\begin{itemize}
    \item Let $(X_0,Y_0)\in \mathcal{R}_1$. The gray `--' symbols in $\mathcal{R}_1\cup\mathcal{R}_2$, obtained from  Lemma~\ref{LemsignCase2}b), indicate that $\mathcal{L}_k(X_t,Y_t)<0$,   for all $t\geq 0$ and so the entire forward orbit must remain below the nullcline $Y=k(X)$. The  orbit cannot remain in $\mathcal{R}_1$ indefinitely, since the component-wise monotonicity given by the direction field
    would imply convergence to an equilibrium, but also prevents convergence to any equilibrium in that region.     Hence, the orbit must eventually enter  $\mathcal{R}_2$, and hence converge to $E_2$.
    \item  Let $(X_0,Y_0)\in \mathcal{R}_3$. If the orbit remains in $\mathcal{R}_3$ indefinitely, by the component-wise monotonicity obtained from the direction field, the orbit must converge to $E_2$.   Otherwise, the  black `+' symbol in $\mathcal{R}_2\cup \mathcal{R}_3$,  obtained from  Lemma~\ref{LemsignCase2}a), indicates that $\mathcal{L}_h(X_t,Y_t)>0$, for all $t\geq 0$. The orbit must therefore remain above the nullcline $Y=h(X)$, and hence must enter $\mathcal{R}_2$ and once again converge to $E_2$. 
\end{itemize}

Thus, all orbits with positive initial conditions converge to  $E_2$. 
By Theorem~\ref{thm:E0} and the direction field, it is clear that $E_0$ is a repeller and $E_1$ is a saddle.
\end{proof}

This theorem implies that if the competitive efficiency of competitor $X$ is negative and the competitive efficiency of competitor $Y$ is positive, then $Y$ is the sole surviving population.  If the sign of the competitive efficiencies are reversed, then population $X$ is the sole surviving population. This result is stated  in the following theorem that can be proven by simply exchanging the parameter indices for $X$ and $Y$.

\begin{theorem}
If $C_{12}>0$ and $C_{21}<0$, then any orbit with initial condition $X_0,Y_0>0$ converges to $E_{1}$,   $E_0$ is a repeller, and  $E_1$ is a saddle.
\end{theorem}


 Fig.~\ref{Fig:C2general}b) provides an example for Case II, that includes the root-curves for the parameter choices $r_1=\frac{1}{2}$, $r_2=\frac{5}{8}$, $K_1=\frac{1}{2}$,  $K_2=2$, and $\alpha_1=\alpha_2=1$.  Although the precise location of the root-curves is not necessary, as only some of the signs of the next-iterate operators were needed to obtain the global dynamics in Theorem~\ref{C2Thm1}, the positions of the   root-curves    are in fact generic for Case II. 
More precisely, the (gray) root-curve  (associated with  nullcline, $Y=k(X)$),  remains above $Y=k(X)$,  intersects $E_2$, and intersects the $X$-axis. The (black) root-curve  (associated with  the nullcline, $Y=h(X))$, remains below $Y=h(X)$, intersects the $Y$-axis, intersects $E_1$, and is  negative,  for $X>K_1$. Above the gray root-curve, the `++' symbols indicate that orbits remain above both nullclines. The direction field in this region implies  that orbits converge to $E_2$ or eventually enter the region between the gray root-curve and the nullcline $Y=k(X)$. There, the  black `+' symbol indicates that an orbit remains above the nullcline $Y=h(X)$  but the gray `--' symbol indicates that they jump below the nullcline $Y=k(X)$. Thus, the orbit enters the region bounded by the nontrivial nullclines. The signs in that region imply that an orbit remains in  that region  and, together with the direction field, imply that the orbit converges to $E_2$.

\subsubsection{Case III: $C_{12}, C_{21}<0$}

In this case,  the nontrivial $X$ and $Y$ nullclines intersect exactly once in the interior of the  first quadrant, and so there exists a unique  coexistence equilibrium $E^*=(X^*,Y^*)$ with $X^*,Y^*>0$. 
  Since the signs of the competitive efficiencies are both negative, 
 \begin{equation}\label{signhkC3}
h(X) - k(X) \quad \begin{cases}\quad  < \quad 0, & \quad \mbox{if }  \,X<X^*,\\
\quad = \quad 0, \quad & \quad \mbox{if }  \,X=X^*,\\
\quad > \quad 0, \quad  & \quad \mbox{if }  \, X>X^*.\end{cases}
 \end{equation}

Since the two nontrivial nullclines intersect in  $\mathbb{R}_+^2$, there are   four regions of interest: 
\begin{alignat*}{2}
     &\mathcal{R}_1=\{(X,Y)\in \mathbb{R}_+^2\colon \, Y< \min\{h(X),k(X)\}\},\quad
     &\mathcal{R}_2=\{(X,Y)\in \mathbb{R}_+^2\colon\,  k(X)\leq Y\leq  h(X)\}\backslash \{E^*\},\\
    & \mathcal{R}_3=\{(X,Y)\in \mathbb{R}_+^2\colon \, Y>\max\{h(X),k(X)\}\},\quad 
     &\mathcal{R}_4=\{(X,Y)\in \mathbb{R}_+^2\colon \, h(X)\leq Y\leq k(X)\}\backslash \{E^*\}.
\end{alignat*}

    

  The proof of the following Lemma relies on \eqref{calcLs} and is provided in  Appendix \ref{PfLemsignCase3}.

\begin{lemma}\label{LemsignCase3}
Assume $C_{12}, C_{21}<0$.
\begin{enumerate}  
\item[a)] $\mathcal{L}_h(X,Y)>0$ for $(X,Y)\in \mathcal{R}_4\cup  \mathcal{R}_{3_4}$,  where $\mathcal{R}_{3_4}:=\mathcal{R}_3\cap \{(X,Y)\colon \, X<X^*\}$. 
\item[b)] $\mathcal{L}_h(X,Y)<0$ for $(X,Y)\in \mathcal{R}_2\cup \mathcal{R}_{1_2}$,  where $\mathcal{R}_{1_2}:=\mathcal{R}_1\cap \{(X,Y)\colon \, X^*<X\}$.
\item[c)] $\mathcal{L}_k(X,Y)>0$ for $(X,Y) \in \mathcal{R}_2 \cup \mathcal{R}_{3_2}$,  where  $\mathcal{R}_{3_2}:=\mathcal{R}_3\cap \{(X,Y)\colon \, Y<Y^*\}$. 
\item[d)] $\mathcal{L}_k(X,Y)<0$ for $(X,Y)\in \mathcal{R}_4 \cup\mathcal{R}_{1_4} $, where $\mathcal{R}_{1_4}:=\mathcal{R}_1\cap \{(X,Y)\colon \, Y^*<Y\}$.
\end{enumerate}
\end{lemma}

Fig.~\ref{Fig:C3general}a)   includes only  the information obtained in Lemma~\ref{LemsignCase3} about the signs of the next-iterate operators required to obtain the global dynamics in this case.     The signs  in $\mathcal{D}_1\cup \mathcal{D}_2$ are   not necessary,  where:
\begin{align}\label{DefregionsD}
\mathcal{D}_1&=:\left\{ (X,Y)\in \mathcal{R}_1
\colon 0<X\leq X^*, \, 0<Y\leq Y^*\right\}\backslash \{E^*\}, \notag \\[2mm]
\mathcal{D}_2&=:\left\{ (X,Y)\in \mathcal{R}_3
\colon \,  X^*\leq X, \, Y^*\leq Y \right\}
\backslash \{E^*\}.
\end{align}

Although, Fig.~\ref{Fig:C3general}b)   is  an example that shows that  the sign of one of the next-iterate operators  can change sign in 
at least one of these regions,  we will show that  no orbit can oscillate between $\mathcal{D}_1$ and $\mathcal{D}_2$. This will be sufficient for us  to determine the global dynamics using the augmented phase portrait.

\begin{proposition}\label{PropinvariantC3}
If $C_{12},C_{21}<0$, then regions $\mathcal{R}_2$ and $\mathcal{R}_4$ are positively  invariant. 
\end{proposition}

\begin{proof}

We use the augmented phase portrait shown in Fig.~\ref{Fig:C3general}a).
Assume $(X_t,Y_t)\in \mathcal{R}_4$. 
The black `+' symbol in $\mathcal{R}_4$, obtained in Lemma~\ref{LemsignCase3}{\it a)}, implies that $\mathcal{L}_h(X_t,Y_t)>0$ so that $(X_{t+1}, Y_{t+1})$ is above the nullcline $Y=h(X)$. The gray `--' symbol in $\mathcal{R}_4$, obtained from Lemma~\ref{LemsignCase3}{\it d)}, implies that $\mathcal{L}_k(X_t,Y_t)<0$ so that $(X_{t+1},Y_{t+1})$ lies below the nullcline $Y=k(X)$. Thus, $(X_{t+1},Y_{t+1})\in \mathcal{R}_4$. 

Next, assume  that $(X_t,Y_t)\in \mathcal{R}_2$. 
Based on the black `--' symbol in $\mathcal{R}_2$, obtained from  Lemma~\ref{LemsignCase3}b), $\mathcal{L}_h(X_t,Y_t)<0$ so that $(X_{t+1},Y_{t+1})$ lies below the nullcline $Y=h(X)$. Also, the gray `+' symbol in that region obtained from Lemma~\ref{LemsignCase3}c), indicates that $\mathcal{L}_k(X_t,Y_t)>0$  so that $(X_{t+1},Y_{t+1})$ lies above the nullcline $Y=k(X)$. Thus, $(X_{t+1},Y_{t+1})\in \mathcal{R}_2$. 
\end{proof}

\begin{figure}[!htb]
    \centering
    a) \hspace{60mm} b)
    
       \includegraphics[scale=0.32]{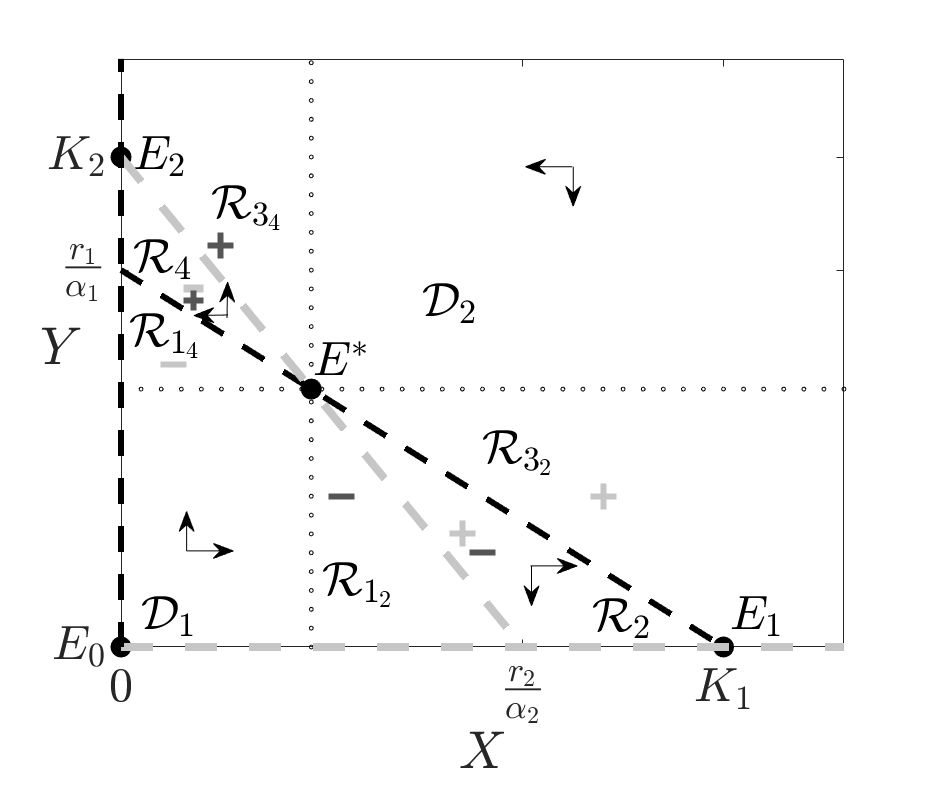}
       \includegraphics[trim= 1cm 0cm 0cm 0cm,clip,scale=0.36]{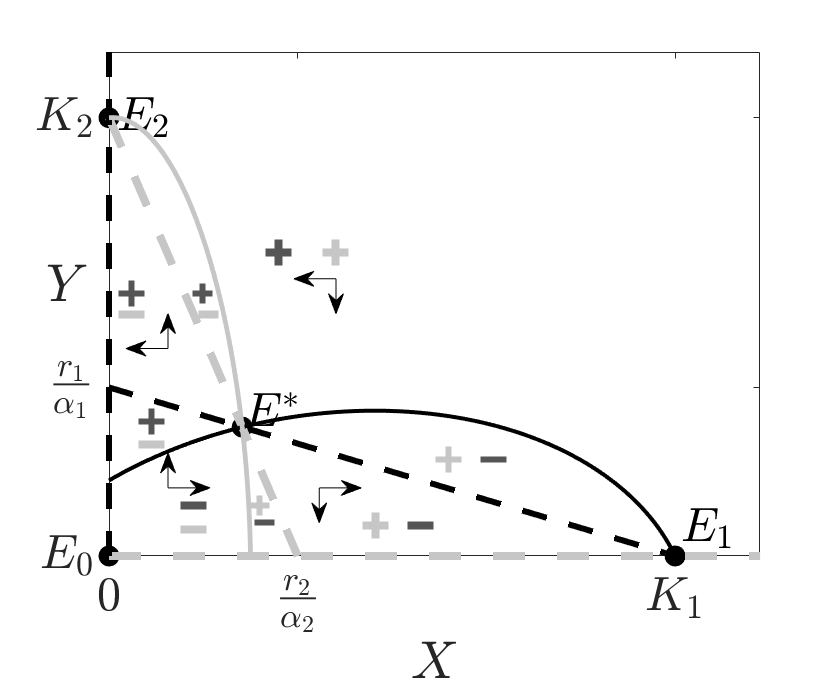}
     \caption{ Augmented phase portraits in  the case  when  $C_{12},C_{21}<0$,
     where  $\mathcal{R}_3= \mathcal{R}_{3_4}\cup\mathcal{D}_2\cup\mathcal{R}_{3_2}$  and $\mathcal{R}_1=\mathcal{R}_{1_2}\cup\mathcal{D}_1\cup\mathcal{R}_{1_4}$.
 In a),  the standard phase portrait
     is augmented by including the  signs  (based on the results in     Lemma~\ref{LemsignCase3}) of the next-iterate operators associated with the nullclines $Y=h(X)$ and $Y=k(X)$, needed to determine the global  dynamics, and is prototypical in this case.    
    In  particular, knowing the signs in regions $\mathcal{D}_1$ and $\mathcal{D}_2$ is not required. 
       In  b), an example  that includes the root-curves is provided   for parameter values, $\alpha_1=1$, $\alpha_2=3$, $r_1=\frac{1}{2}$, $r_2=2$, $K_1=2$,  and $K_2=1.3$. 
}
\label{Fig:C3general}
\end{figure}


 Using the augmented phase portrait in Fig.~\ref{Fig:C3general}a), we   show the following result by arguing that there cannot be a `++' region in $\mathcal{D}_1$ (see the details in Appendix \ref{Pfbox1}).

\begin{lemma}\label{box1}
Let $C_{12}, C_{21}<0$. If $(X,Y)\in \mathcal{D}_1$, then $(F(X,Y),G(X,Y))\notin \mathcal{D}_2$.
\end{lemma}

\begin{theorem}\label{GAC3}
If $C_{12}, C_{21}<0$, then every orbit with $X_0,Y_0>0$ converges to $E_{2}$, $E_{1}$, or $E^*$. Moreover, $E_1$ and $E_2$ are locally asymptotically stable and $E^*$ is unstable.
\end{theorem}

\begin{proof}
Since, by Proposition~\ref{PropinvariantC3}, regions $\mathcal{R}_2$ and $\mathcal{R}_4$ are positively invariant, and the direction field in $\mathcal{R}_2$ implies that orbits that enter $\mathcal {R}_2$ converge to $E_1$ and orbits that enter $\mathcal{R}_4$ converge to $E_2$,
it suffices to show that all solutions either converge to $E^*$ or eventually enter $\mathcal{R}_2\cup \mathcal{R}_4$. 

We use Fig.~\ref{Fig:C3general}a)  to discuss the global dynamics of orbits with initial conditions outside of $\mathcal{R}_2\cup \mathcal{R}_4$.

\begin{itemize}
    \item Let $(X_t,Y_t)\in \mathcal{R}_{1_4}$. The gray `--' symbol in that region,  derived from  Lemma~\ref{LemsignCase3}{\it d)}, indicates  $(X_{t+1}, Y_{t+1})$ is below the line $Y=k(X)$. By the direction field in this region,
 $(X_{t+1},Y_{t+1})\in \mathcal{R}_4\cup \mathcal{R}_{1_4}$. If an  orbit were to  remain in $\mathcal{R}_{1_4}$ indefinitely, then it would have to converge to an equilibrium. However, the direction field excludes the convergence to the only equilibrium in this region, $E^*$.  Thus, the orbit must enter $\mathcal{R}_4$ and  then converges to $E_{2}$. 
\item Let $(X_t,Y_t)\in \mathcal{R}_{3_4}$.  
The black `+' symbol in that region,  obtained from Lemma~\ref{LemsignCase3}{\it a)}, implies that $(X_{t+1},Y_{t+1})$  is above the line $Y=h(X)$. Together with the direction field, this implies that $(X_{t+1},Y_{t+1})\in \mathcal{R}_{3_4}\cup \mathcal{R}_4$. 
If an orbit remains in  $\mathcal{R}_{3_4}$ indefinitely, then it could only converge to $E_2$. Otherwise, it enters $\mathcal{R}_4$ and converges to $E_2$.
\item Let $(X_t,Y_t)\in \mathcal{R}_{1_2}$. Based on the black `--' symbol in that region, obtained from
 Lemma~\ref{LemsignCase3}{\it b)}, $(X_{t+1},Y_{t+1})$ is below the line $Y=h(X)$. By the direction field, $(X_{t+1},Y_{t+1})\in \mathcal{R}_2\cup \mathcal{R}_{1_2}$. If an orbit were to remain in $\mathcal{R}_{1_2}$, 
 it would have to converge to an equilibrium. However, the direction field in this region prevents the convergence to the only equilibrium in this region, $E^*$. Thus, the  orbit must  enter $\mathcal{R}_2$ and hence converge to $E_{1}$.  
\item Let $(X_t,Y_t)\in \mathcal{R}_{3_2}$. Based on the gray `+' symbol in that region, obtained from Lemma~\ref{LemsignCase3}{\it c)},   $(X_{t+1},Y_{t+1})$ lies above the nullcline $Y=k(X)$. Thus, $(X_{t+1},Y_{t+1})\in \mathcal{R}_{3_2}\cup \mathcal{R}_2$. If an orbit were to remain indefinitely in $\mathcal{R}_{3_2}$, then it would have to converge to an equilibrium. The direction field in this region reveals that such an orbit would have to converge to $E_1$. Otherwise, the orbit must  enter $\mathcal{R}_2$ and also converge to $E_{1}$. 

\item Let $(X_t,Y_t)\in \mathcal{D}_1$. If the orbit remains indefinitely in 
$\mathcal{D}_1$, then it converges to $E^*$. Otherwise, by Lemma~\ref{box1},
there exists $T>0$ such that 
$(X_T,Y_T)\in \mathcal{D}_1$ and $(X_{T+1},Y_{T+1})\in (0,K_1)\times (0,K_2) \backslash \mathcal{D}_2$. Thus, one of the previous cases applies and hence the orbit converges  to $E_{1}$ or $E_{2}$.

\item  Let $(X_t,Y_t)\in \mathcal{D}_2$. First, assume that $(X_0,Y_0)\in (X^*,K_1)\times  (Y^*,K_2).$ 
 If the orbit remains  in 
$\mathcal{D}_2$ indefinitely, then  from the direction field, it must converge to $E^*$.  Otherwise, the orbit enters one of the other regions and one of the previous cases applies, so that the orbit converges  to $E_{1}$ or $E_{2}$.
  
\end{itemize}

Thus,  any orbit with positive initial conditions converges to one of the equilibria, $E^*$, $E_1$, or $E_2$.

Without calculating the eigenvalues of the Jacobian, we can also conclude that $E_2$ is locally asymptotically stable because  any orbit with initial condition $(X_0,Y_0)\in \mathcal{R}_4\cup \mathcal{R}_{3_4}$  converges to $E_2$ and the convergence is monotone in a neighbourhood of $E_2$.  Similarly, since any orbit with initial condition $(X_0,Y_0)\in  \mathcal{R}_2\cup \mathcal{R}_{3_2}$  converges to $E_1$, $E_1$ is locally asymptotically stable. Since the coexistence equilibrium $E^*$ is on the boundary of  all of these regions, $E^*$ is unstable.  
\end{proof}

We were able to prove the global dynamics based on the augmented phase portrait in Fig.~\ref{Fig:C3general}a), without knowing the precise location of the root-curves. If specific parameter values were however chosen, then the root-curves can be obtained numerically and the signs of the next-iterate operators associated with each of the nullclines can be obtained for the entire first quadrant, see Fig.~\ref{Fig:C3general}b). 
For the specific example with 
$\alpha_1=1$, $\alpha_2=3$, $r_1=\frac{1}{2}$, $r_2=2$, $K_1=2$,  and $K_2=1.3$, the root-curves were obtained.  The black solid curve in  Fig.~\ref{Fig:C3general}b),   represents the root-curve associated with the nonrivial $X$-nullcline and the gray solid curve is the root-curve associated with the nontrivial $Y$-nullcline.    Once the root-curves are included, the corresponding signs of the next-iterate operators  can be added in every region.  Fig.\ref{Fig:C3general}b),   highlights that orbits in the $\mathcal{D}_2$ region cannot jump into $\mathcal{D}_1$. For example, an the next iterate of an orbit in $\mathcal{D}_2$ where  black and gray `+' symbols are, must remain above both nullclines. Similarly, the next iterate of an orbit in $\mathcal{D}_1$, where  black and gray `--' symbols are, must  remain below both nullclines and can therefore not enter $\mathcal{D}_2$. 
Since the proof of Theorem \ref{GAC3} was only based on Fig.~\ref{Fig:C3general}a), not all signs of the next-iterate operators are necessary to determine the global dynamics.

\subsubsection{Case IV: $C_{12}, C_{21}>0$}

In this case,   as in Case III, the nontrivial $X${\color{red}-} and $Y${\color{red}-}nullclines intersect exactly once in the interior of the  first quadrant, and so there exists a unique  coexistence equilibrium $E^*=(X^*,Y^*)$ with $X^*,Y^*>0$. 
 However, since  $C_{12},C_{21}>0$, 
\begin{equation}\label{signhkC4}
    h(X)-k(X)\, \quad \begin{cases} \quad
    > \quad 0, \quad & \quad \mbox{if }  \,X<X^*,\\
   \quad  =\quad 0, & \quad \mbox{if }  \,X=X^*,\\
    \quad < \quad 0, & \quad \mbox{if }  \,X>X^*.
    \end{cases} 
\end{equation}

We can again divide the first quadrant into the four regions, 
\begin{alignat*}{2}
    &\mathcal{R}_1:=\{(X,Y)\in \mathbb{R}_+^2\colon Y<\min\{h(X),k(X)\}\},\,
    &\mathcal{R}_2:=\{(X,Y)\in \mathbb{R}_+^2\colon h(X)\leq Y\leq k(X)\}\backslash \{E^*\},\\
    &\mathcal{R}_3:=\{(X,Y)\in \mathbb{R}_+^2\colon Y>\max\{h(X),k(X)\}\},\,
    &\mathcal{R}_4:=\{(X,Y)\in \mathbb{R}_+^2\colon k(X)\leq Y\leq h(X)\}\backslash \{E^*\}.\\
\end{alignat*}

  The proof of the next Lemma is provided in Appendix \ref{PfLemsignCase4}.

\begin{lemma}\label{LemsignCase4}
 Assume $C_{12}, C_{21}>0$.
\begin{enumerate} 
\item[{\it a)}] $\mathcal{L}_h(X,Y)<0$ for $(X,Y)\in \mathcal{R}_4\cup  \mathcal{R}_{1_4}$,  where $\mathcal{R}_{1_4}:=\mathcal{R}_1\cap \{(X,Y)\colon \, Y>Y^*\}$. 
\item[{\it b)}] $\mathcal{L}_h(X,Y)>0$ for $(X,Y)\in \mathcal{R}_2\cup \mathcal{R}_{3_2}$,  where $\mathcal{R}_{3_2}:=\mathcal{R}_3\cap \{(X,Y)\colon \, Y<Y^*\}$.
\item[{\it c)}] $\mathcal{L}_k(X,Y)>0$ for $(X,Y) \in \mathcal{R}_4 \cup \mathcal{R}_{3_4}$,  where  $\mathcal{R}_{3_4}:=\mathcal{R}_3\cap \{(X,Y)\colon \, X<X^*\}$. 
\item[{\it d)}] $\mathcal{L}_k(X,Y)<0$ for $(X,Y)\in \mathcal{R}_2 \cup\mathcal{R}_{1_2} $, where $\mathcal{R}_{1_2}:=\mathcal{R}_1\cap \{(X,Y)\colon \, X>X^*\}$.
\end{enumerate}
\end{lemma}

Fig.~\ref{Fig:C4general}a) includes the information obtained in Lemma~\ref{LemsignCase4} about the signs of the next-iterate operators. As for Case III, we do not need to determine the sign of the next-iterate operators in regions  $\mathcal{D}_1$ and $\mathcal{D}_2$ to determine the global dynamics.

\begin{proposition}\label{invariantC4}
If $C_{12},C_{21}>0$, then regions $\mathcal{R}_2$ and $\mathcal{R}_4$ are positively invariant.
\end{proposition}

\begin{proof}
We use the  augmented phase portrait shown in Fig.~\ref{Fig:C4general}a).
Let $(X_t,Y_t)\in \mathcal{R}_2$. Based on the black `+' symbol in $\mathcal{R}_2$,  obtained from  Lemma~\ref{LemsignCase4}{\it b)}, $\mathcal{L}_h(X_t,Y_t)>0$, so that $(X_{t+1},Y_{t+1})$ remains  above the  nullcline $Y=h(X)$. The gray `--' symbol in $\mathcal{R}_2$, derived from Lemma~\ref{LemsignCase4}{\it d)}, indicates that $\mathcal{L}_k(X_t,Y_t)<0$. Thus,  $(X_{t+1},Y_{t+1})$ remains below  the nullcline $Y=k(X)$. Thus, $(X_{t+1},Y_{t+1})\in \mathcal{R}_2$.

Next, let $(X_t,Y_t)\in \mathcal{R}_4$. The black `--' symbol in this region, derived from  Lemma~\ref{LemsignCase4}{\t a)}, indicates that $\mathcal{L}_h(X,Y)<0$ and so $(X_{t+1},Y_{t+1})$ remains below the nullcline $Y=h(X)$. 
Also, the gray `+' symbol in this region, derived from  Lemma~\ref{LemsignCase4}{\it c)},  indicates that $\mathcal{L}_k(X_t,Y_t)>0$, so that $(X_{t+1},Y_{t+1})$ remains above the nullcline $Y=k(X)$. Thus,\\ $(X_{t+1},Y_{t+1}) \in \mathcal{R}_4$, completing the proof.
\end{proof}

\begin{figure}[!ht]
    \centering
    a) \hspace{75mm} b)
    
       \includegraphics[trim= .5cm 0cm 0cm 1cm, clip,scale=0.3]{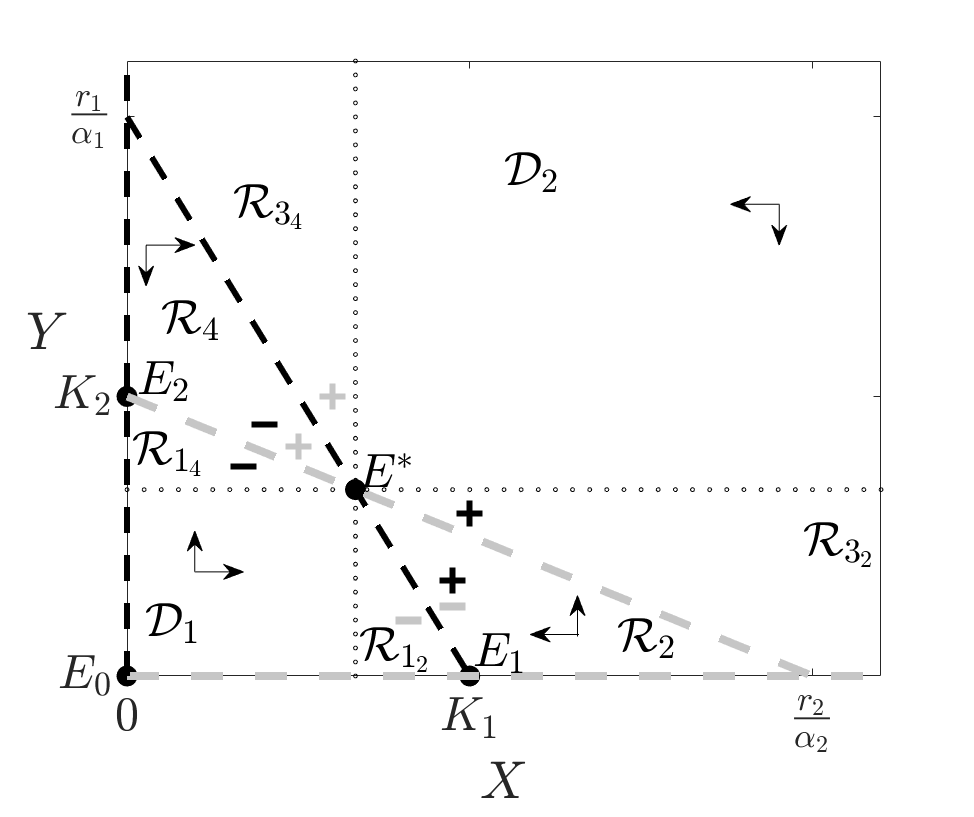}
       \includegraphics[trim= .4cm 0cm 0cm 0cm, clip,scale=0.315]{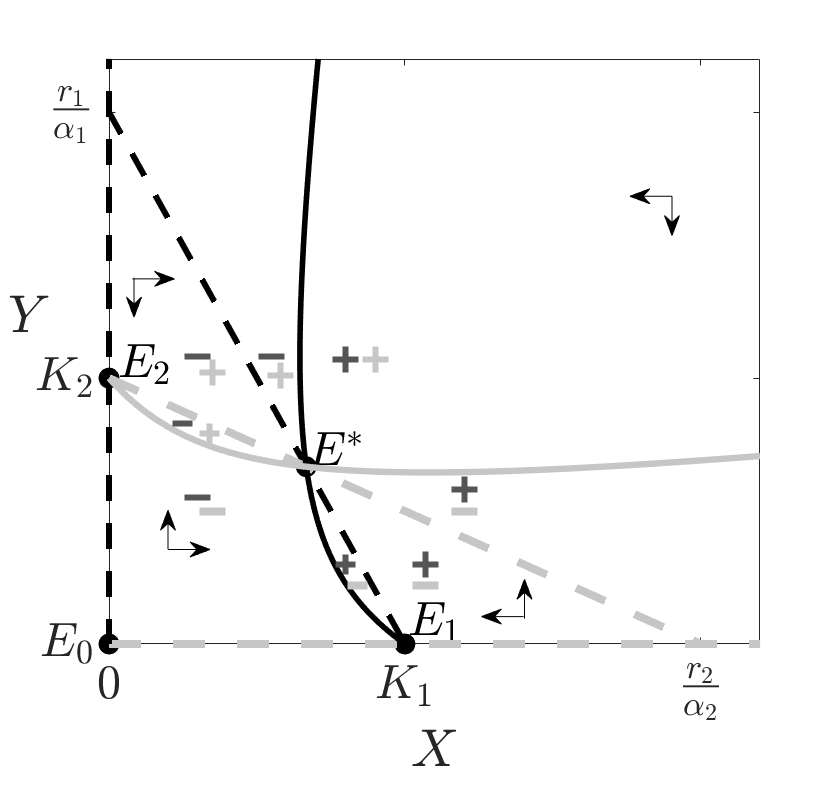}
     \caption{ 
     Augmented phase portraits in  the case  when $C_{12}, C_{21}>0$, where $\mathcal{R}_1=\mathcal{D}_1\cup \mathcal{R}_{1_4}\cup \mathcal{R}_{1_2}$ and $\mathcal{R}_3=\mathcal{D}_2\cup \mathcal{R}_{3_2}\cup \mathcal{R}_{3_4}$. In a), the standard phase portrait  is augmented by including only the  signs (based on the results in Lemma~\ref{LemsignCase4}) of the next-iterate operators associated with the nullclines $Y=h(X)$ and $Y=k(X)$, necessary to determine the global stability of $E^*$ and is prototypical for this case.  
     An example including root-curves is shown in b)  with parameters, $\alpha_1=\alpha_2=1$, $r_1=r_2=2$, $K_1=K_2=1$.
 } 
    \label{Fig:C4general}
\end{figure}

 As for Lemma~\ref{box1}, the proof of the following result that shows that there is no 
 `++' region in $\mathcal{D}_1$, is 
 obtained using the augmented phase portrait,  in this case, shown in Fig.~\ref{Fig:C4general}a) (see details in Appendix \ref{Pfbox2}).

\begin{lemma}\label{box2}
Let $C_{12}, C_{21}>0$. If $(X,Y)\in \mathcal{D}_1$, then $(F(X,Y),G(X,Y))\notin \mathcal{D}_2$, where $\mathcal{D}_1$ and $\mathcal{D}_2$ are defined in \eqref{DefregionsD}.
\end{lemma}

\begin{theorem}\label{GAC4}
If $C_{12},C_{21}>0$, then  $E^*$ is globally asymptotically stable with respect to  orbits with $X_0,Y_0>0$. Furthermore, $E_0$ is a repeller, and $E_1$ and $E_2$ are saddles.
\end{theorem}

\begin{proof}
We use the augmented phase portrait in Fig.~\ref{Fig:C4general}a) to show that all solutions with  positive initial conditions  converge to $E^*$. First note that  by Proposition \ref{invariantC4}, $\mathcal{R}_i$, $i=1,2$, are positively invariant  and that the direction field in these regions  implies that any  orbit that enters  either of these two regions converges to $E^*$. 

\begin{itemize}
    \item Let $(X_t,Y_t)\in \mathcal{R}_{1_4}$.  
 The black `--' symbol, obtained from Lemma~\ref{LemsignCase4}{\it a)}, implies that $(X_{t+1},Y_{t+1})$ must remain below the nullcline $Y=h(X)$. The direction field tells us that $(X_{t+1},Y_{t+1})\in \mathcal{R}_{1_4}\cup \mathcal{R}_4$. If an orbit were to remain in $\mathcal{R}_{1_4}$ indefinitely, then it must converge to an equilibrium. However, the direction field in this regions prevents the convergence to the only equilibria in this region.  Hence, there exists $T>0$ such that  $(X_T,Y_T)\in \mathcal{R}_{4}$ and then the orbit must converge to $E^*$. 
 \item Let $(X_t,Y_t)\in \mathcal{R}_{3_4}$. Based on the gray `+' symbol, obtained from  Lemma~\ref{LemsignCase4}{\it c)}, $(X_{t+1}, Y_{t+1})$ must remain above the nullcline $Y=k(X)$.  From  the direction field, it follows that $(X_{t+1},Y_{t+1})\in \mathcal{R}_{3_4}\cup \mathcal{R}_4$. If an orbit were to remain in $\mathcal{R}_{3_4}$ indefinitely, it must converge to an equilibrium. However, the direction field in this region prevents the convergence to the only equilibrium in this region,  $E^*$. Hence, there exists $T>0$ such that  $(X_T,Y_T)\in \mathcal{R}_{4}$ and o the orbit must converge to $E^*$.

    \item Let $(X_t,Y_t)\in \mathcal{R}_{1_2}$. The gray `--' symbol, based on 
    Lemma~\ref{LemsignCase4}{\it d)}, reveals that $(X_{t+1}, Y_{t+1})$ must remain below the nullcline $Y=k(X)$. Thus, with the direction field, it follows that $(X_{t+1},Y_{t+1})\in \mathcal{R}_{1_2}\cup \mathcal{R}_2$. If an orbit were to remain in $\mathcal{R}_{1_2}$ indefinitely, then it must converge to an equilibrium. However, the direction field in this region prevents the convergence to the only two equilibria, $E^*$ and $E_1$. Hence, there exists $T>0$ such that  $(X_T,Y_T)\in \mathcal{R}_{2}$ and the orbit converges to $E^*$. 
    \item Let $(X_t,Y_t)\in \mathcal{R}_{3_2}$. The black `+' symbol, based on 
    Lemma~\ref{LemsignCase4}{\it b)}, implies  that $(X_{t+1},Y_{t+1})$ remains above the nullcline $Y=h(X)$. Together with the direction field,  $(X_{t+1},Y_{t+1})\in \mathcal{R}_{3_2}\cup \mathcal{R}_2$. If an orbit were to remain in $\mathcal{R}_{3_2}$ indefinitely, then it must converge to an equilibrium. However, the direction field in this region prevents the convergence to the only equilibrium $E^*$. Thus, there exists $T>0$ such that  $(X_T,Y_T)\in \mathcal{R}_{2}$ and the orbit converges to $E^*$. 

\item Let $(X_t,Y_t)\in \mathcal{D}_1$. If the orbit remains   in 
$\mathcal{D}_1$ indefinitely, then it must converge to $E^*$. Otherwise, by Lemma~\ref{box2}, there exists $T>0$ such that 
$(X_T,Y_T)\in \mathcal{D}_1$. However, $(X_{T+1},Y_{T+1})\in \left(0,\frac{r_2}{\alpha_2}\right)\times\left(0, \frac{r_1}{\alpha_1}\right)$ and  so one of the previous cases apply. 

\item Let $(X_t,Y_t)\in \mathcal{D}_2$. If the orbit remains in 
$\mathcal{D}_2$ indefinitely, then it converges to $E^*$. Otherwise,  the orbit enters one of the other regions, where   one of the previous cases apply  and the orbit must converge to $E^*$. 
\end{itemize}

Thus, any orbit with positive initial conditions converges to $E^*$ and the convergence is eventually monotone. Hence, $E^*$ is globally asymptotically stable.  That $E_0$ is a repeller and $E_1$ and $E_2$ are saddles also follow from Theorem~\ref{thm:E0} and the augmented phase portrait.
\end{proof}

Theorem \ref{GAC4} does not require the sign of the next-iterate operator in all regions of the first quadrant. However, for specific parameter values, one can graph the root-curves associated with each nullcline, see Fig.~\ref{Fig:C4general}b). Once the root-curves are obtained, the signs of the next-iterate operators can immediately be included in the phase portrait.


\section{Extensions and  Limitations}\label{sec:ext_limits}

While the previous sections focused on the introduction of the augmented phase portrait and how to use it in the analysis of the discrete competition model \eqref{Compete}, the method can easily be used for other planar maps and provides an  elementary tool to obtain  information about the local and  global dynamics of solutions. However, just as for the phase plane approach used for the analysis of planar ordinary differential equations, the  augmented phase  portrait  has its limitations. Some of these are discussed in this section.

\subsection{Example: Ricker  Competition Model}

A popular alternative to \eqref{Compete} is the competitive Ricker map:
\begin{equation}\label{Rickerexample}
    X_{t+1}=X_te^{K-X_t-aY_t}, \qquad \quad
    Y_{t+1}=Y_te^{L-bX_t-Y_t},
\end{equation}
with initial conditions $X_0,Y_0\geq 0$, where $K,L>0$ represent the carrying capacities of competitor $X$ and $Y$, respectively. Here $a,b>0$ describe the competitive factor for population $X$ and $Y$, respectively. For $0<K,L<1$, the  theory of monotone flows  was applied, allowing for conclusions regarding the global dynamics given the local stability of equilibria \cite{Smith1998}. In \cite{CabralBalreira2014}, \eqref{Rickerexample} was revisited and conditions were provided for the global stability of the coexistence equilibrium  under  different  restrictions on the parameters.  Nevertheless, the conjecture that  for \eqref{Rickerexample}, local asymptotic stability always implies global asymptotic stability \cite{Luis2011} remains an open problem.  While the augmented phase plane method cannot be used to prove the conjecture, it can be used to identify positively invariant regions and therefore the global dynamics of orbits entering these regions. In turn, the augmented phase portrait also determines regions where solutions might oscillate.  This might be helpful  to prove or disprove the conjecture.  

\begin{figure}[!bth]
    a) \hspace{8cm} b) \hspace{5cm}
    \centering
\includegraphics[trim= 0cm 0cm 1cm 0cm,clip,scale=0.37]{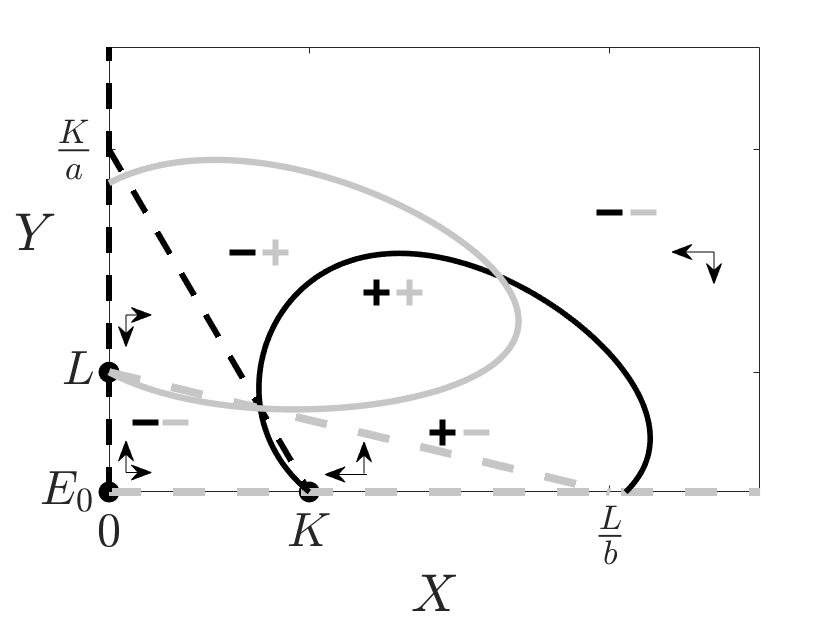}
\includegraphics[trim= 0cm 0cm 2cm 0cm,clip,scale=0.37]{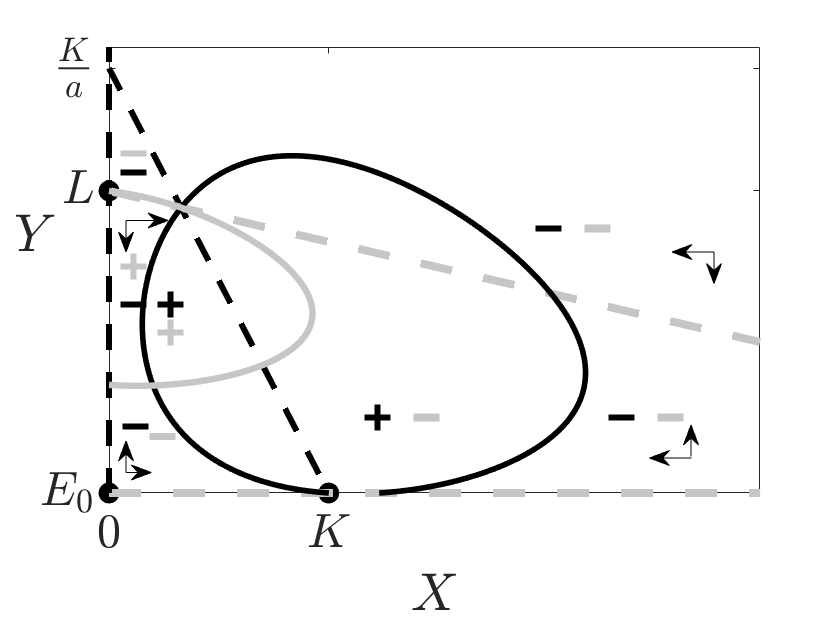}
    \caption{Augmented phase portraits for \eqref{Rickerexample}. In a),  $K=L=0.6$, $a=0.35$, and $b=0.4$. The augmented phase portrait can be used to prove the global asymptotic stability of the coexistence equilibrium.  In b), $K=0.9$, $L=1.6$, $a=0.4$, and $b=0.3$. The augmented phase portrait cannot rule out orbits jumping back and forth between the '++' region below both nullclines and the '-- --' region above them.}
    \label{Fig:Ricker}
\end{figure}


For the specific model parameters chosen in Fig.~\ref{Fig:Ricker}, the root-curves  were  obtained numerically, using the built-in function ``fimplicit'' in Matlab.  In contrast to all of the root-curves we have seen thus far, e.g., Fig.~\ref{Fig:C4general}b),  the root-curves  in Fig.~\ref{Fig:Ricker} are neither functions in $X$ nor $Y$. In this case, the sign of the next-iterate operators depend on whether it is evaluated at a point that is  ``inside'' or ``outside'' of the region bounded by the associated root-curve. For points inside (outside) the region bounded by a root-curve, the corresponding next-iterate operator  is positive (negative), indicating that the next iterate will lie above (below) its associated nullcline. 

From the augmented phase portrait  in Fig.~\ref{Fig:Ricker}a), it is possible to  determine that the coexistence equilibrium is globally asymptotically stable with respect to the interior of the first quadrant.  Based on the signs of the next-iterate operators, the augmented phase portrait identifies two regions as positively invariant: i) \, the triangular region  bounded by the nontrivial nullclines and the $X$-axis with left-corner $K$ and right-corner $\frac{L}{b}$, and 
ii)\, the triangular region bounded by the nontrivial nullclines and the $Y$-axis with the lower $Y$-value $L$ and upper value $\frac{K}{a}$. Orbits entering either one of these two regions remain there, and, due to the direction field, must converge to the coexistence equilibrium. 
The signs of the next-iterate operators together with the direction field can also be used to argue that any orbit in the interior of the first quadrant must enter either i) or ii), and therefore converge to $E^*$.

The coexistence equilibrium for the parameter choice for Fig.~\ref{Fig:Ricker}b) is locally asymptotically stable, as the eigenvalues of the Jacobian are within the unit-circle. However, in this particular example, the augmented phase portrait cannot even be used to  determine the local  asymptotic stability of the coexistence equilibrium. 
It however identifies regions of interest. For example, an orbit could oscillate between the small region  containing  the black and gray `+' symbols, and the region with the black and gray `--' symbols above both nullclines.   
Furthermore, none of the regions bounded by nullclines is positively invariant. This can be immediately recognized by noting that in all four component-wise monotone regions, there exists at least one root-curve associated with a nullcline that partially lies in  this region. This causes a change in the sign of the corresponding next-iterate root operator.  




\subsection{Example: Model with  Mutualism}

We consider the following  example  involving mutualism:
\begin{equation}\label{eq:Mutualism}
X_{t+1}=\frac{(a+bY_t) X_t}{A+BX_t}, \qquad Y_{t+1}=\frac{(c+dX_t)Y_t}{C+DY_t},
\end{equation}
with initial conditions $X_0,Y_0\geq 0$ and positive parameters.

Augmented phase portraits for \eqref{eq:Mutualism} are shown in Fig~\ref{fig:Mutualism} for two different parameter choices.  
Although, for the choice of parameters in Fig.~\ref{fig:Mutualism}a), the root-curves are not unique, this is not what prevents determining that $E^*$ is globally asymptotically stable with respect the interior of the first quadrant.  It is,  that we cannot rule out orbits oscillating indefinitely between regions $\mathcal{D}_1$ and $\mathcal{D}_2$ without converging to $E^*$.  What the augmented phase portrait does tell us is that the basin of attraction of $E^*$ is contained in the union of all of the regions that  have one `+' and one '-' symbol, the part of the region on the left containing two `+' symbols, where $Y\leq Y^*$, and the part of the region on the bottom-right containing two `--' symbols where $X<X^*$.

From the augmented phase portrait in Fig.~\ref{fig:Mutualism}b), we can conclude that $E^*$ is globally asymptotically stable with respect to initial conditions $(X_0,Y_0)\in (0,6]\times (0,6]$. The problematic regions  $\mathcal{D}_1$ and $\mathcal{D}_2$ in Fig.~\ref{fig:Mutualism}a), are now detached from the equilibrium and are each separate curves outside $[0,6]\times [0,6]$.

\begin{figure}[h!]
\centering 
a) \hspace{7.5cm} b)

  \begin{multicols}{2}
 \includegraphics[trim=1cm 0cm 2cm 0cm,clip,scale=.25]{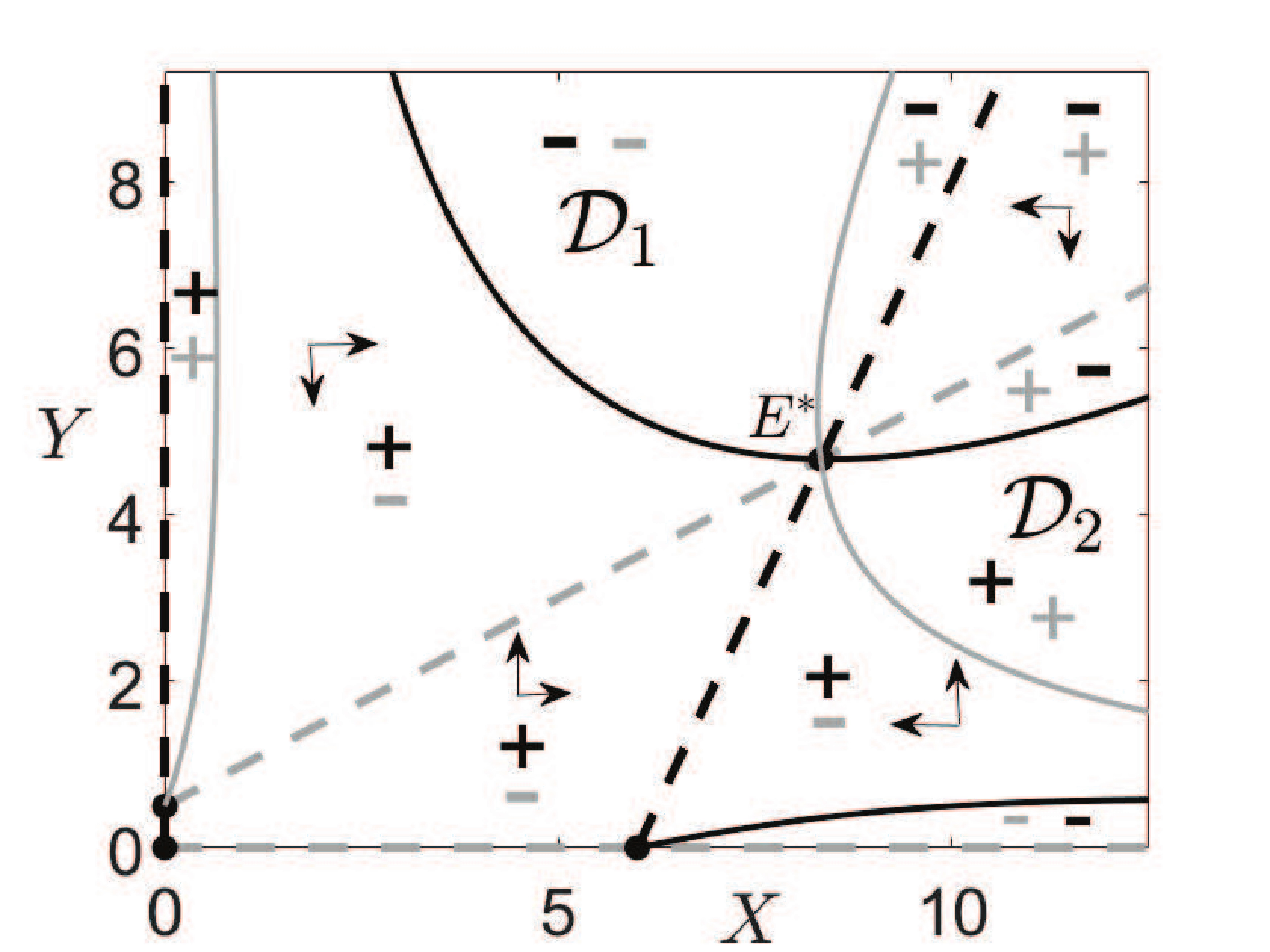}
 \columnbreak
\includegraphics[trim=0cm 0cm 0cm 0cm,clip,scale=0.34]{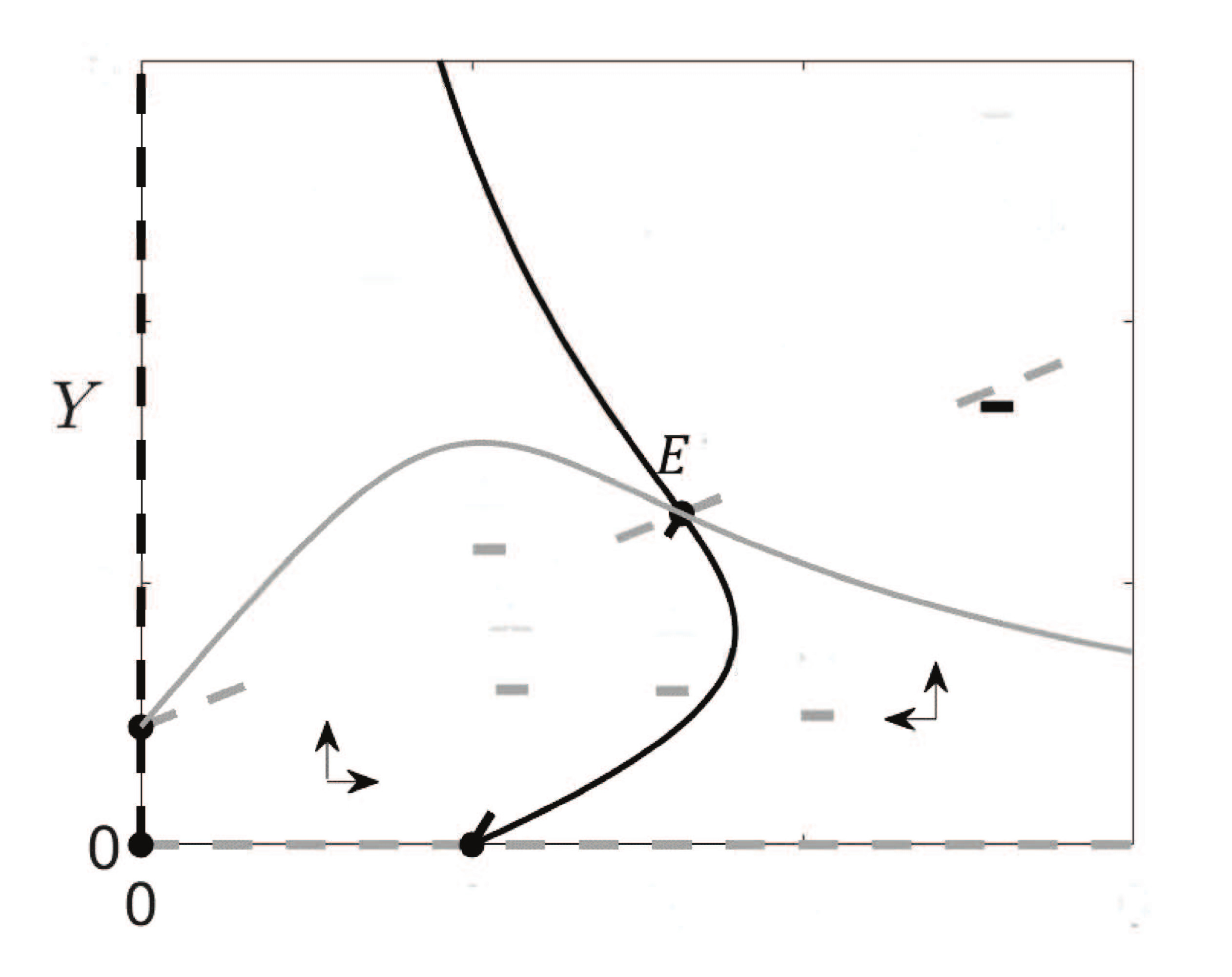}
 \end{multicols}
 \caption{   The augmented phase portrait for \eqref{eq:Mutualism}.  In a),  $a=16$, $b=1$, $A =4$, $B=2$, $c=4$, $d=1$, $C=3$, $D=2$. There are two root-curves associated with each of the nontrivial nullclines. All orbits with initial conditions outside of $\mathcal{D}_1\cup \mathcal{D}_2$, converge to $E^*$.  
 In b), $a=8$, $b=1$, $A =4$, $B=2$, $c=4.8$, $d=1$, $C=3$, $D=2$. All orbits with initial conditions $(X_0,Y_0)\in (0,6]\times (0,6]$, converge to $E^*$.}
 \label{fig:Mutualism}
 \end{figure}

\subsection{Example: Predator--Prey Model}

In \cite{StWoBo1},  we  derived and analysed the discrete predator--prey model:
\begin{equation}\label{LV1}
X_{t+1}=\frac{(1+r)X_t}{1+\frac{r}{K}X_t+\alpha Y_t}, \qquad Y_{t+1}=\frac{(1+\gamma X_t)Y_t}{1+d},   
\end{equation}
with initial conditions $X_0,Y_0\geq 0$, 
where all parameters are positive and 
$X$ and $Y$ denote the prey and predator populations. 

In \cite{StWoBo1},  the root-curve associated with the (nontrivial) prey nullcline was used to discuss the global dynamics of solutions of \eqref{LV1}.
 In the case when no coexistence equilibrium exists, an augmented phase portrait, as in Fig.~\ref{Fig:PP}a),  was used to determine the global asymptotic stability of the prey-only equilibrium $E_K=(K,0)$. When a coexistence equilibrium $E^*$ exists, as in Fig.~\ref{Fig:PP}b),  it was shown  that the augmented phase portrait excludes the existence of prime period 2 and 3 orbits. The global asymptotic stability of $E^*$, whenever it is locally asymptotically stable (i.e., $d<\gamma K<1+2d$), remains a conjecture.

\begin{figure}[h!]

 a) \hspace{7cm} b) \hspace{3cm}
 
 \includegraphics[scale=0.35]{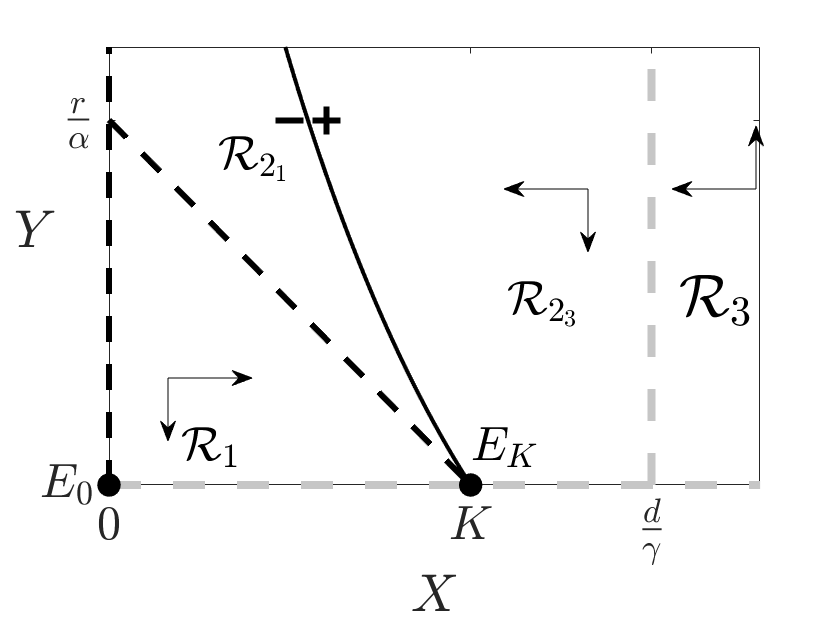}
 \includegraphics[scale=0.35]{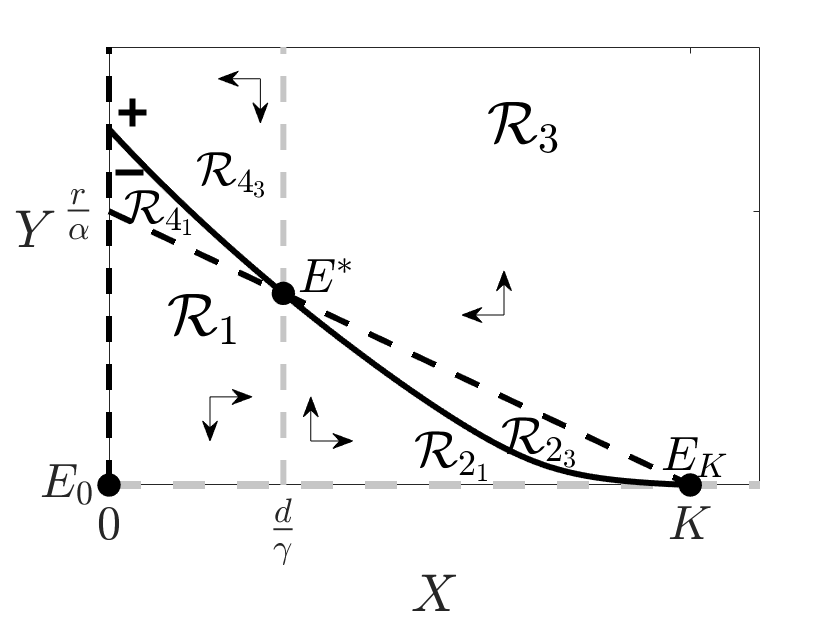}
    \caption{Augmented phase portraits for the predator--prey model derived in  \cite{StWoBo1}.  In a), where $K\gamma<d$,  the prey-only equilibrium $E_K$ is globally asymptotically stable. In b), where $d<\gamma K$, there is a unique interior equilibrium $E^*$. Orbits either converge to $E^*$ eventually monotonically 
    or they cycle around $E^*$ indefinitely visiting each of the regions  $\mathcal{R}_i$, \, $i=1,2,3,4,$ at least once in each cycle.}
    \label{Fig:PP}
\end{figure}

Just as for the continuous analogue of \eqref{LV1} ($x'=rx\left(1-\frac{x}{K}\right)-\alpha xy$, $y'=y(-d+\gamma x)$), phase plane analysis alone is  not enough to obtain a complete picture of the global dynamics when a coexistence equilibrium exists. For the continuous model, additional arguments are needed, including the application of the Dulac criterion, to rule out period orbits, and the Poincar\'{e}-Bendixson Theorem. Even though  the configuration of the standard phase portraits for the continuous predator-prey model is the same as in Fig.~\ref{Fig:PP}, unlike the continuous model for which no periodic orbits are possible and the existence of the coexistence equilibrium implies it is globally asymptotically stable,  in \cite{StWoBo1} it was shown that for the discrete model the coexistence equilibrium undergoes a Neimark-Sacker bifurcation when $\gamma K = 1+2d$ and loses it stability.

\section{Conclusion}
 
 We describe an elementary approach for analyzing planar discrete maps that can provide information about the global dynamics.  Standard phase plane analysis has not been very effective in this context, since unlike in the case of planar systems defined by smooth differential equations,   orbits of discrete maps can jump over nullclines, as shown in Fig.~\ref{Fig:Jump}.  To overcome this drawback, we introduce the next-iterate operators associated with the nullclines and their associated root-sets and root-curves.
Knowing the sign of the next-iterate operators in a region of the phase plane tells us on which side of the nullcline the operator is associated with, the next iterate will lie.
By providing examples, we show that it is sometimes possible to determine the global dynamics of planar maps, by augmenting the standard phase portrait by including the signs of the next-iterate operators, where required.  We then call the standard phase portrait that includes these signs, the {\it augmented phase portrait}.

 In Section~\ref{sec:compete},  we showed how to use the augmented phase plane to determine the global dynamics of a  well-studied two species competition model \eqref{Compete}.  We provided a more elementary approach, compared to the use of the theory of monotone flows, to show that the relative values of the competitive efficiencies completely determine the global dynamics, just as in the case of the analogous continuous model \eqref{Compete_cont}. Using the augmented phase portrait, we were also able to determine the local stability of all of the equilibria without having to resort to linearization (i.e., finding the eigenvalues of the Jacobian at the equilibrium), even in one case when linearization would have been inconclusive. We were also able to find the invariant and positively invariant regions and then use the direction field within these regions to conclude convergence of orbits once they enter one of these regions.

We also   discuss some extensions and limitations of the augmented phase portrait in Section~\ref{sec:ext_limits} by considering three examples:  a Ricker competition model, a model involving species that display mutualistic behavior, and a  predator-prey model. The complexity of root-sets for the Ricker competition model was illustrated in Fig.~\ref{Fig:Ricker}.  We provided one set of parameters for which use of the augmented phase portrait could be used to  determine the global dynamics completely and one that illustrated that there can be problematic regions in the phase portrait. Next, we addressed a model involving mutualism. Fig.~\ref{fig:Mutualism}a)  illustrated that the root-curves do not have to be unique. However, it was not the non-uniqueness of the root-curves that  prevented determining the global dynamics from the augmented phase portrait. Instead it was the existence of a  `++' region below both nullclines and a `-- --' region above both nullclines that, along with the  direction field, did not allow ruling out orbits  oscillating  between these two regions. It is also important to note that although Figs.~\ref{Fig:Ricker}a) and \ref{fig:Mutualism}b) also have a `++' region and a `-- --' region, these do not cause a problem due to the direction field in those regions. 
Finally, for the predator-prey model, it is possible    to use the augmented phase portrait to determine the  asymptotic outcome for all orbits  in the case that there is no coexistence equilibrium, and in particular  prove that the  prey-only equilibrium  is globally asymptotically stable when it is locally asymptotically stable. The augmented phase portrait also showed that if an orbit does not converge to the coexistence equilibrium, the orbit  cycles around it  and must visit four different regions at least once in every cycle,  thus ruling out  prime period 2 and period 3 orbits.

 In ongoing research, we continue to explore whether this elementary approach, i.e., using the augmented phase portrait,  can be used in other contexts to determine different global properties of discrete planar models such as delay difference equations  and general rational maps.
\medskip

\noindent
{\it Acknowledgement}:  The research of Gail S. K. Wolkowicz was partially supported by a Natural Sciences and Engineering Research Council of Canada (NSERC) Discovery grant with accelerator supplement.

\appendix

\section{Appendix}   



\subsection{Proof of (\ref{calcLs})}\label{PfcalcLs}

Substituting  the expression for the $X$ nullcline given in  
\eqref{eq:Linex} in $\mathcal{L}_h(X,Y)$,
 we have
\begin{align*}
    \mathcal{L}_h(X,Y)&=G(X,Y)-h(F(X,Y)) = \frac{(1+r_2)Y}{1+\frac{r_2}{K_2}Y+\alpha_2X}-\frac{r_1}{\alpha_1 K_1}\left(K_1-\frac{(1+r_1)}{1+\frac{r_1}{K_1}X+\alpha_1Y}\right)\notag \\
&=\frac{ N_h(X,Y)}{\alpha_1 (K_1 + r_1 X + \alpha_1 K_1 Y) (K_2 + \alpha_2 K_2 X + r_2 Y)}, 
\end{align*}
where
\begin{equation*}
N_h(X,Y)=a_0(X)+a_1(X)Y+a_2Y^2=A_0(Y)+A_1(Y)X+A_2X^2,
\end{equation*}
with 
\begin{align}
a_0(X)\quad &= \quad -r_1 K_2 (K_1 - X) (1 + \alpha_2 X), \notag\\
a_1(X)\quad&= \quad -r_1 r_2 (K_1 - X) - 
 \alpha_1 K_2 (-r_1 (1 + r_2) X + K_1 (-1 + r_1 - r_2 + \alpha_2 r_1 X)), \notag\\
a_2\qquad \quad&=\quad \alpha_1 K_1 (\alpha_1 K_2 (1+r_2)-r_1 r_2),\label{CoefLh}\\
    A_0(Y)\quad&=\quad K_1 (1 + \alpha_1 Y) (-r_1 r_2 Y + K_2 (-r_1 + \alpha_1 (1 + r_2) Y)),\notag \\
    A_1(Y)\quad&=\quad r_1 (r_2 Y + K_2 (1 + \alpha_1 (1 + r_2) Y - \alpha_2 (K_1 + \alpha_1 K_1 Y))),\notag \\
    A_2\qquad \quad&=\quad \alpha_2K_2 r_1. \notag
\end{align}
It follows that  
\qquad $\mathcal{L}_h(X,Y)=0\quad \iff \quad Y=r_{h_i}(X) \quad \mbox{or } \quad X=R_{h_i}(Y), \quad i=1,2,$
where
\begin{equation}\label{rh}
    r_{h_1}(X):= \begin{cases} \frac{-a_1(X) + \sqrt{a_1^2(X)-4a_0(X)a_2}}{2a_2}, & a_2\neq 0,\\
    -\frac{a_0(X)}{a_1(X)}, & a_2=0,\end{cases} \quad     r_{h_2}(X):= \begin{cases} \frac{-a_1(X) - \sqrt{a_1^2(X)-4a_0(X)a_2}}{2a_2}, & a_2\neq 0,\\
    -\frac{a_0(X)}{a_1(X)}, & a_2=0,\end{cases}
\end{equation}
and
\begin{equation}\label{Rh}
    R_{h_1}(Y):=\frac{-A_1(Y)+ \sqrt{A_1^2(Y)-4A_0(Y)A_2}}{2A_2}, \quad R_{h_2}(Y):=\frac{-A_1(Y) - \sqrt{A_1^2(Y)-4A_0(Y)A_2}}{2A_2}.
\end{equation}

Substituting  the expression for the $Y$ nullcline given in   \eqref{Liney} in  $\mathcal{L}_k(X,Y)$, we have
\begin{align*}
    \mathcal{L}_k(X,Y)&=G(X,Y)-k(F(X,Y))= \frac{(1+r_2)Y}{1+\frac{r_2}{K_2}Y+\alpha_2X}-\frac{K_2}{r_2}\left(r_2-\alpha_2\frac{(1+r_1)}{1+\frac{r_1}{K_1}X+\alpha_1Y}\right)\notag \\
&=\frac{N_k(X,Y)}{r_2 (K_1 + r_1 X + \alpha_1 K_1 Y) (K_2 + \alpha_2 K_2 X + r_2 Y)},  
\end{align*}
with 
\begin{equation*}
N_k(X,Y)=b_0(X)+b_1(X)Y+b_2Y^2=B_0(Y)+B_1(Y)X+B_2X^2,
\end{equation*}
where
\begin{align}
    b_0(X) \quad &= \quad K_2^2 (1 + \alpha_2 X) (K_1 (  \alpha_2 (1 + r_1) X-r_2)-r_1r_2X), \notag\\
b_1(X) \quad &= \quad K_2 r_2 (r_1 X + K_1 (1 + \alpha_2 (1 + r_1) X - \alpha_1 (K_2 + \alpha_2 K_2 X))), \notag\\
    b_2 \qquad \quad &  =\quad \alpha_1 r_2 K_1 K_2,\label{CoefLk}\\
    B_0(Y)\quad &=\quad -K_1 K_2 r_2 (K_2 - Y) (1 + \alpha_1 Y),\notag \\
    B_1(Y)\quad & = \quad K_2 (r_1 r_2 (Y-K_2) + 
   \alpha_2 K_1 ((1 + r_1) r_2 Y + K_2 (1 + r_1 - r_2 (1 + \alpha_1 Y)))),\notag \\
   B_2 \qquad \quad &= \quad \alpha_2 K_2^2 (\alpha_2 K_1 (1 + r_1) - r_1 r_2).\label{CoefLkiny}
\end{align}
 We therefore have 
\begin{equation*}
\mathcal{L}_k(X,Y)=0\qquad \iff \qquad Y=r_{k_i}(X) \quad \mbox{or } \quad X=R_{k_i}(Y), \qquad i=1,2,
\end{equation*}
where
\begin{equation}\label{rk}
    r_{k_1}(X):=\frac{-b_1(X)+\sqrt{b_1^2(X)-4b_0(X)b_2}}{2b_2}, \quad r_{k_2}(X):=\frac{-b_1(X)-\sqrt{b_1^2(X)-4b_0(X)b_2}}{2b_2}
\end{equation}
and 
\begin{equation}\label{Rk}
    R_{k_1}(Y):=\begin{cases}
\frac{-B_1(Y)+\sqrt{B_1^2(Y)-4B_0(Y)B_2}}{2B_2}, & B_2\neq 0,\\
-\frac{B_0(Y)}{B_1(Y)}, & B_2=0,\end{cases} \,\,
     R_{k_2}(Y):=\begin{cases}
\frac{-B_1(Y)- \sqrt{B_1^2(Y)-4B_0(Y)B_2}}{2B_2}, & B_2\neq 0,\\
-\frac{B_0(Y)}{B_1(Y)} & B_2=0.\end{cases}
\end{equation}

\subsection{Proof of Lemma \ref{LhLkCase1}}\label{PfLhLkCase1}

\begin{proof}
Substituting $ K_2=\frac{r_1}{\alpha_1}$ and $K_1=\frac{r_2}{\alpha_2}$ in \eqref{CoefLh}, we have $ a_2=\frac{\alpha_1 r_1 r_2}{\alpha_1}>0$
and therefore, by \eqref{rh},    
$r_{h_1}(X)=\frac{r_1( r_2 - \alpha_2 X)}{\alpha_1 r_2}=h(X)$
and  
$r_{h_2}(X)=\frac{ -(\alpha_2 X+1)}{\alpha_1}<0$.
Thus, there is a unique positive  root-curve, $r_h(X)=r_{h_1}(X)=h(X),$  associated with the positive $X$ nullcline, $Y=h(X),$ that is defined for $0<X<\frac{r_2}{\alpha_2}=K_1.$
Since the next-iterate operator associated with $Y=h(X)$ only changes sign at this root-curve, we have with \eqref{LhX0},
\begin{equation*}
    \mathcal{L}_h(X,Y)\quad \begin{cases}
 \quad    <\quad 0,\quad & \quad \mbox{if }  \,(X,Y)\in \mathcal{R}_1, \\
 \quad    >\quad 0, \quad &  \quad \mbox{if }  \,(X,Y)\in \mathcal{R}_2 ,\end{cases} 
\end{equation*}
so that \eqref{signC1hb} follows.

Substituting $ K_2=\frac{r_1}{\alpha_1}$ and $K_1=\frac{r_2}{\alpha_2}$ in \eqref{CoefLkiny}, we have
$B_2= \frac{\alpha_2 r_1^2 r_2}{\alpha_1^2}>0$,
and therefore, by \eqref{Rk},
$X=R_{k_1}(Y)=\frac{r_1}{\alpha_2}\left(1-\frac{Y}{K_2}\right)=k^{-1}(Y) $
and 
$ X=R_{k_2}(Y)=\frac{-(\alpha_1 Y +1)}{\alpha2}<0$.
Thus, there is again a unique positive root-curve $X=R_{k_1}(Y)=k^{-1}(Y)$ that is positive for $0<Y< K_2$ associated with  the positive $Y$ nullcline, $Y=k(X)$.
Hence, the next-iterate root operator associated with the positive $Y$ nullcline only changes sign  at the nullcline $Y=k(X)$, and with \eqref{LkX0}, it follows that
\begin{equation*}
\mathcal{L}_k(X,Y)\quad \begin{cases} \quad < \quad 0,\quad & \, \mbox{if} \quad (X,Y)\in \mathcal{R}_1,\\
\quad >\quad 0, \quad &\, \mbox{if} \quad (X,Y)\in \mathcal{R}_2,\end{cases}
\end{equation*}
 and \eqref{signC1kb} follows.
\end{proof}

\subsection{Proof of Lemma \ref{rootintersectD1}}\label{PfrootintersectD1}
\begin{proof}
First, recall that if $C_{12}\cdot C_{21}>0$, then $E^*\in \mathcal{E}$.
Clearly, $E^*$ is in the intersection of the sets.  Let $(\bar{X},\bar{Y})\in S_k\cap S_h$. Then, by Remark \ref{intersectgen}, $F(\bar{X}, \bar{Y})=X^*$ and $G(\bar{X},\bar{Y})=Y^*$. That is, 
\begin{align*}
F(\bar{X},\bar{Y})-F(X^*,Y^*)=\frac{(1+r_1)}{m_1(\bar{X},\bar{Y})m_1(X^*,Y^*)}\left(\bar{X}-X^*+\alpha_1(\bar{X}Y^*-X^*\bar{Y})\right)\\ 
G(\bar{X}, \bar{Y})-G(X^*,Y^*)
=\frac{(1+r_2)}{m_2(\bar{Y},\bar{X})m_2(Y^*,X^*)}\left(\bar{Y}-Y^*+\alpha_2 (X^*\bar{Y}-\bar{X}Y^*)\right),
\end{align*}
where $m_i(u,v)=1 +\frac{r_i}{K_i}u+\alpha_iv>0$, for $i=1,2$. 

We use proof by contradiction to show that  no point is also in  $(\bar{X},\bar{Y})\in \{(X,Y)\colon 0<X\leq X^*, 0<Y\leq Y^*\}\backslash E^*$  or $(\bar{X},\bar{Y})\in \{(X,Y)\colon\,  X\geq X^*, \,  Y\geq Y^*\}\backslash E^*$.
Since  $\bar{X}\leq X^*$ and $\bar{Y}\leq Y^*$ or  $\bar{X}\geq X^*$ and $\bar{Y}\geq Y^*$, but $(\bar{X}, \bar{Y})\neq E^*$, we have without loss of generality $\bar{X}\neq X^*$,
\begin{align*}
0 >  (\bar{X}-X^*)& =  ~ \alpha_1(X^* \bar{Y}-\bar{X} Y^*),\\
0  \geq (\bar{Y}-Y^*)& =  - \alpha_2(X^* \bar{Y}-\bar{X} Y^*),
\end{align*}
yielding a contradiction. 
\end{proof}

\subsection{Proof of Lemma \ref{LemsignCase2}}\label{PfLemsignCase2}
\begin{proof}
{\it a)}  Assume that $(X,Y)\in \mathcal{R}_2\cup \mathcal{R}_3$. 
By \eqref{CoefLh},
$$a_2=\alpha_1K_1\left(\alpha_1K_2(1+r_2)-r_1r_2\right)\stackrel{C_{12}<0}{\geq} \alpha_1 K_1\left(\alpha_1 \frac{r_1}{\alpha_1}(1+r_2)-r_1r_2\right)=\alpha_1 K_1 r_1>0.$$

 We consider two sub-cases: {\it a)(i)} $X\geq K_1$ and {\it a)(ii)} $X<K_1$.

{\it a)(i)} Assume that  $X\geq K_1$.  By \eqref{CoefLh}, 
$a_1(X)$ is linear in $X$, and since $C_{21}>0$, 
$$a_1'(X)=r_1 (r_2 + \alpha_1 K_2 (1 - \alpha_2 K_1 + r_2))> r_1 (r_2 + \alpha_1 K_2) >0$$
and therefore  
 $$a_1(X)\geq a_1(K_1)=\alpha_1 K_1 K_2 (1  + r_2 + r_1 (r_2- \alpha_2 K_1) )>\alpha_1 K_1 K_2 (1  + r_2) >0.$$ Thus, $\frac{-a_1(X)}{2 a_2} <0$.
 Also, $a_0(X)a_2>0$, and   
 since the term under the radical in  \eqref{rh} is less than $a_1^2(X)$, we have  $Y=r_{h_1}(X)<0$ and
 $Y=r_{h_2}(X)<0$.  Since the sign of the associated next-iterate operator can only change sign at a root-curve,  
 $\mathcal{L}_h(X,Y)$ has the same sign for any  $X\geq K_1$. 
Since,  
 $$\lim_{Y\to \infty} \mathcal{L}_h(K_1,Y)=\lim_{Y\to \infty}G(K_1,Y)-h(F(K_1,Y))=\frac{1+r_2}{r_2}K_2-h(0)>K_2-\frac{r_1}{\alpha_1} > 0,$$
$\mathcal{L}_h(X,Y)>0$ for all $X\geq K_1$.

 {\it a)(ii)} Assume that $X\in (0,K_1)$. Since,  by \eqref{CoefLh}, $a_2>0$ and $a_0(X)<0$, by \eqref{rh}, $Y=r_{h_2}(X)<0$ and $Y=r_{h_1}(X)>0$, and so  $S_h$ is uniquely determined for $X\in (0,K_1)$ by  the root-curve $Y=r_{h_1}(X)$. 
 Since there is no coexistence equilibrium, by Lemma~\ref{Lemnoint},  $Y=r_{h_1}(X)$ cannot intersect the nullcline $Y=h(X)$ for any $0<X<K_1$, so that $Y=r_{h_1}(X)$ must lie either entirely above or entirely below  $Y=h(X)$. Since $h(X)<k(X)$, for $0<X<K_1$, by Lemma~\ref{usefulhk}a), $\mathcal{L}_h(X,h(X))>0$. By \eqref{LhX0}, $\mathcal{L}_h(X,0)<0$. Thus, $\mathcal{L}_h(X,Y)$ must have already changed sign in $\mathcal{R}_1$,
 and so  
 $Y=r_{h_1}(X)$ must lie below $Y=h(X)$. Thus,   $\mathcal{L}_h(X,Y)>0$,
 for all $0<X<K_1$, such that  $X\in\mathcal{R}_2\cup\mathcal{R}_3$. 
 
   By {\it a)(i)} and {\it a(ii)},   $\mathcal{L}_h(X,Y)>0$, for all $(X,Y)\in \mathcal{R}_2\cup\mathcal{R}_3.$

{\it b)} Assume that $(X,Y)\in \mathcal{R}_1\cup \mathcal{R}_2$.  
The sign of $b_0(X)$, defined in \eqref{CoefLk}, is the same as the sign of the  factor
\begin{align*}
\tilde{b}_0=( K_1 (  \alpha_2 (1 + r_1) X-r_2)-r_1r_2X)
 &=     X (  
 \alpha_2 K_2( 1+r_1) -r_1 r_2)  -  r_2 K_1, 
 \end{align*}
a linear function of $X$.   $\tilde{b}_0(0)=-r_2 K_1<0$, and since $C_{21}>0$, $\tilde{b}_0\left(\frac{r_2}{\alpha_2}\right)= \frac{r_1 r_2 K_2^2(r_2 + 1)( \alpha_2 K_1 - r_2)}{\alpha_2}<0$.
 Hence, $b_0(X)<0$, for all $X\in\left[0,\frac{r_2}{\alpha_2}\right]$.  By \eqref{CoefLk}, $b_2>0$  so that by \eqref{rk}, $Y=r_{k_1}(X)>0$  and $Y=r_{k_2}(X)<0$ for all $X\in \left(0,\frac{r_2}{\alpha_2}\right)$. Hence, only one root-curve is positive for  $X\in \left(0,\frac{r_2}{\alpha_2}\right)$. This implies that $\mathcal{L}_k(X,Y)$ can change sign at most once in $\mathcal{R}_1\cup\mathcal{R}_2.$
By Lemma~\ref{usefulhk}a), since $h(X)<k(X)$, $\mathcal{L}_k(X,k(X))<0$  for all $X\in\left(0,\frac{\alpha_2}{r_2}\right).$ 
Since  $k(X)$ is decreasing and $F(X,Y)$ is increasing in $X$,  
$$\mathcal{L}_k(X,0)\stackrel{\eqref{LkX0}}{=}-k(F(X,0))<-k(F(K_1,0))<-k(K_1)<0.$$
Therefore,  
the sign of the next-iterate operator associated with $Y=k(X)$ did not change sign for $X\in \left(0,\frac{\alpha_2}{r_2}\right)$ and so
$\mathcal{L}_k(X,0)<0$, for all $X\in\mathcal{R}_1\cup\mathcal{R}_2$.
 \end{proof}

\subsection{Proof of Lemma \ref{LemsignCase3}}\label{PfLemsignCase3}
\begin{proof}
 First consider {\it a)} and {\it b)}.
Since $C_{12}<0$, $K_2>\frac{r_1}{\alpha_1}$, and so the sign of $a_2$ is given by the sign of 
$$\alpha_1K_2(1+r_2)-r_1r_2>\alpha_1\frac{r_1}{\alpha_1}(1+r_2)-r_1r_2=r_1>0.$$
Since $a_2>0$   and   $a_0(X)<0$, for all $X\in (0,K_1)$,
we have by \eqref{rh}, $r_{h_2}(X)<0$  and  $r_{h_1}(X)>0$, for all $X\in (0,K_1)$. Hence, the function $Y=r_{h_1}(X)$ determines $S_h$ uniquely  in the regions considered in {\it a)} and{\it  b)}.
\begin{enumerate}
  \item[{\it a)}]  Assume that $(X,Y)\in  \mathcal{R}_4\cup  \mathcal{R}_{3_4}$.
By \eqref{LhX0}, $\mathcal{L}_h(0,X)<0$ for all $X\in (0,X^*)$ and by  \eqref{signhkC3} and  Lemma~\ref{usefulhk}{\it a)},   $\mathcal{L}_h(X,h(X))>0$ for all $X\in(0,X^*)$.
Therefore, the  associated next-iterate operator, $\mathcal{L}_h(X,Y)$ must already have changed sign and become positive below $Y=h(X)$
when $X\in(0,X^*)$. Thus, the unique positive root-curve, $Y=r_{h_1}(X)$, must lie below the line $Y=h(X)$  for all $X\in (0,X^*)$.  We conclude that
$\mathcal{L}_h(X,Y)>0$, for all $(X,Y)\in \mathcal{R}_4\cup  \mathcal{R}_{3_4}$.

\item[{\it b)}] Assume that $(X,Y)\in\mathcal{R}_2\cup \mathcal{R}_{1_2}$. 
By Lemma~\ref{Lemnoint},  the only points where $Y=r_{h_1}(X)$ and the nullcline $Y=h(X)$ intersect are the equilibrium points $E_1$ and $E^*$.
  We again use \eqref{signhkC3}, Lemma~\ref{usefulhk}{\it b)}, and \eqref{LhX0} to conclude that  $\mathcal{L}_h(X,h(X))<0$ for $X^*<X<K_1$.
 Since, by \eqref{LhX0}, $\mathcal{L}_h(X,0)<0$ for all $X\in (X^*,K_1)$, we have by continuity that the next-iterate operator associated with $Y=h(X)$ must change sign above the nullcline $Y=h(X)$ for $X\in (X^*,K_1)$. Thus, the root-curve $Y=r_{h_1}(X)$ must lie above the nullcline $Y=h(X)$ for $X^*<X<K_1$ and hence $\mathcal{L}_h(X,Y)<0$ for all $(X,Y)\in \mathcal{R}_2\cup \mathcal{R}_{1_2}$.

   \hspace{.25in} Next consider {\it c)} and {\it d)}. Since $C_{21}<0$, $K_1>\frac{r_2}{\alpha_2}$, and so 
 $$B_2=\alpha_2 K_2^2 (\alpha_2 K_1 (1 + r_1) - r_1 r_2)>\alpha_2  K_2^2 r_2>0.$$
 Since $B_2>0$   and   $B_0(Y)<0$ for all $Y\in (0,K_2)$,
we have by \eqref{rk} that $R_{k_2}(Y)<0$  and  $R_{k_1}(Y)>0$ for all $Y\in (0,K_2)$. Hence, for $Y\in (0,K_2)$,  the function $X=R_{k_1}(Y)$ 
determines $S_k$  uniquely  in the regions considered in c) and d).

\item[{\it c)}]  Assume that $(X,Y)\in \mathcal{R}_2 \cup \mathcal{R}_{3_2}$.
 For $Y\in (0,Y^*)$, $X>X^*$, so that by \eqref{signhkC3} and Lemma~\ref{usefulhk}{\it b)}, $\mathcal{L}_k(X,k(X))>0$. 
Since, by \eqref{LkY0}, $\mathcal{L}_k(0,Y)<0$, this implies that the sign of the next-iterate operator associated with $Y=k(X)$ must have changed sign to the left of $X=k^{-1}(Y)$, which in turn means that the root-curve $X=R_{h_1}(Y)$ must be below the line $Y=k(X)$ for $Y\in(0,Y^*)$.  Hence, $\mathcal{L}_k(X,Y)>0$ for all $(X,Y)\in \mathcal{R}_2 \cup \mathcal{R}_{3_2}$.

\item[{\it d)}] Assume that $(X,Y)\in \mathcal{R}_4 \cup\mathcal{R}_{1_4}$.
By Lemma~\ref{Lemnoint},  the only points where $X=R_{h_1}(Y)$ and the nullcline $Y=h(X)$ intersect at the equilibrium points $E_2$ and $E^*$.
For $Y\in (Y^*,K_2)$ and $X\in (0,X^*)$ by \eqref{signhkC3} and 
Lemma~\ref{usefulhk}{\it a)}, $\mathcal{L}_k(X,k(X))<0$ for $X\in(0,X^*)$. 
Since $X=R_{h_1}(Y)$ uniquely determines $S_k$ for $Y\in (0,K_2)$ and $\mathcal{L}_k(0,Y)<0$ for $Y\in (0,K_2)$, the next-iterate operator associated with $Y=k(X)$ did not change sign to the left of $X=k^{-1}(Y)$.  Therefore, the next-iterate operator   changes sign above the nullcline $Y=k(X)$ (i.e., to the right of $X=k^{-1}(X)$). Hence, $\mathcal{L}_k(X,Y)<0$ for all $(X,Y)\in \mathcal{R}_4 \cup\mathcal{R}_{1_4}$.
\end{enumerate}
\vspace{-.85cm}
\end{proof}

\subsection{Proof of Lemma~\ref{box1}}\label{Pfbox1}

\begin{proof}
 The proof relies on the signs of the next-iterate operators indicated in the augmented phase portrait shown in Fig.~\ref{Fig:C3general}a), based on  Lemma~\ref{LemsignCase3}. By Fig.~\ref{Fig:C3general}a), $\mathcal{R}_{1_4}$ has a gray `--' symbol, so that $\mathcal{L}_k(X,Y)<0$ in that region. In $\mathcal{R}_{1_2}$, the black `--' symbol indicates that $\mathcal{L}_h(X,Y)<0$ in that region. An orbit in $\mathcal{D}_1$ can only enter $\mathcal{D}_2$ if there is a sub-region of  $\mathcal{D}_1$ with a black `+' symbol and a gray `+' symbol, as the orbit would have to jump over both nullclines. Since the sign can only change at root-curves and both root-curves cannot cross within $\mathcal{D}_1$ by Lemma~\ref{rootintersectD1},  it suffices to show that the  root-curve associated with the nullcline $Y=k(X)$ remains below the root-curve associated with the nullcline $Y=h(X)$ in $\mathcal{D}_1$. In the proof of  Lemma~\ref{LemsignCase3}c) and d),  it was shown that $\mathcal{S}_k$ is determined uniquely by the root-curve $X=R_{k_1}(Y)$ for $Y\in (0,K_2)$.  Also, $B_2>0$, $B_0(0)=-K_1 K_2^2 r_2<0$, and therefore $X=R_{k_1}(Y)$ intersects the $X$-axis at a value  $X\in(0,\frac{r_2}{\alpha_2}]$, noting the  gray `+' symbols in region $\mathcal{R}_2\cup \mathcal{R}_{3_2}$.

 In the proof of  Lemma~\ref{LemsignCase3}a) and b), it was shown that  $\mathcal{S}_h$ is determined uniquely by the root-curve $Y=r_{h_1}(X)$ for $X\in (0,K_1)$. By \eqref{CoefLh}, $a_2>0$ and since $a_0(0)=-K_1 K_2 r_1<0$,  it follows by \eqref{rh} that $r_{h_1}(0)>0$, and so $Y=r_{h_1}(X)$ intersects the $Y$-axis  at a value $Y\in(0,\frac{\alpha_1}{r_1})$, noting the black `+' symbols in $\mathcal{R}_4\cup\mathcal{R}_{3_4}$.
    Since, by Lemma \ref{rootintersectD1}, the root-curves do not intersect in $\mathcal{D}_1$, $X=R_{k_1}(Y)$ must remain to the right of (below) the root-curve $Y=r_{h_1}(X)$. Hence, no sub-region of $\mathcal{D}_1$ exists where both $\mathcal{L}_h(X,Y)$ and $\mathcal{L}_k(X,Y)$ are positive. 
\end{proof}

\subsection{Proof of Lemma \ref{LemsignCase4}}\label{PfLemsignCase4}
\begin{proof}
 First consider {\it a)} and {\it b)}. Since $C_{12}>0$, $K_2<\frac{r_1}{\alpha_1}$, and so  the sign of $A_0(Y)$ is given by the sign of
\begin{align*}
   -r_1r_2 Y+K_2(-r_1+\alpha_1(1+r_2)Y)
&=-r_2Y(r_1-\alpha_1K_2)-K_2(r_1-\alpha_1Y)\\
&<-K_2(r_1-\alpha_1Y).
\end{align*}
Thus, $A_0(Y)<0$ for  all $Y\in \left(0,\frac{r_1}{\alpha_1}\right)$. Since $A_2>0$, we have by \eqref{Rh} that $R_{h_2}(Y)<0$ and  $R_{h_1}(Y)>0$ for all $Y\in \left(0,\frac{r_1}{\alpha_1}\right)$. Hence, the function $X=R_{h_1}(Y)$ determines $S_h$ uniquely  in the regions considered in a) and b). Further note that for $Y\in \left(0,\frac{r_1}{\alpha_1}\right)$, we have by \eqref{LhY0},
\begin{equation}\label{LhY0C4}
\mathcal{L}_h(0,Y)=G(0,Y)-\frac{r_1}{\alpha_1}=\frac{(1+r_2)Y}{1+\frac{r_2}{K_2}Y}-\frac{r_1}{\alpha_1} = \frac{(1+r_2-\frac{r_1}{\alpha_1K_2}r_2)Y-\frac{r_1}{\alpha_1}}{1+\frac{r_2}{K_2}Y}<0,
\end{equation}
because $Y<\frac{r_1}{\alpha_1}$ and 
$1+r_2-\frac{r_1}{\alpha_1K_2}r_2<1$.

\begin{enumerate}
  \item[{\it a)}]  Assume that $(X,Y)\in  \mathcal{R}_4\cup  \mathcal{R}_{1_4}$.
 By  \eqref{signhkC4} and  Lemma~\ref{usefulhk}{\it b)},   $\mathcal{L}_h(X,h(X))<0$ for all $X\in(0,X^*)$,  i.e., $\mathcal{L}_h(h^{-1}(Y),Y)<0$ for all $Y \in \left(Y^*,\frac{r_1}{\alpha_1}\right)$.   Since by \eqref{LhY0C4},  $\mathcal{L}_h(0,Y)<0$. This implies that the  next-iterate operator $\mathcal{L}_h(X,Y)$ did not change sign  between the $Y$-axis and the nullcline $Y=h(X)$ for $Y\in \left(Y^*,\frac{r_1}{\alpha_1}\right)$. Therefore, the associated root-curve $X=R_{h_1}(Y)$ must be on the right of $X=h^{-1}(Y)$ for  $Y\in \left(Y^*,\frac{r_1}{\alpha_1}\right)$ and so 
$\mathcal{L}_h(X,Y)<0$ for all $(X,Y)\in \mathcal{R}_4\cup  \mathcal{R}_{1_4}$.

\item[{\it b)}] Assume that $(X,Y)\in\mathcal{R}_2\cup \mathcal{R}_{3_2}$. 
By Lemma~\ref{Lemnoint},  the only points where $X=R_{h_1}(Y)$ and the nullcline $Y=h(X)$ intersect are the equilibrium points $E_1$ and $E^*$.
  We again use \eqref{signhkC3} and Lemma~\ref{usefulhk}{\it a)} to conclude that  $\mathcal{L}_h(h^{-1}(Y),Y)=\mathcal{L}_h(X,h(X))>0$ for all $X^*<X<K_1$,  i.e., $\mathcal{L}_h(h^{-1}(Y),Y)>0$ for $0<Y<Y^*$.
 Since, by \eqref{LhY0C4}, $\mathcal{L}_h(0,Y)<0$ for $Y\in (0,Y^*)\subset \left(0,\frac{r_1}{K_1}\right)$, we have by continuity of  $\mathcal{L}_h(X,Y)$ that the next-iterate operator associated with $Y=h(X)$  must  change sign between the $Y$-axis and the nullcline $Y=h(X)$. Hence, the root-curve $X=R_{h_1}(Y)$ must lie to the left of $Y=h(X)$ for $Y \in (0,Y^*)$. Hence, $\mathcal{L}_h(X,Y)>0$ for all $(X,Y)\in \mathcal{R}_2\cup \mathcal{R}_{3_2}$.
 \end{enumerate}

Next consider {\it c)} and {\it d)}. Since $C_{21}>0$, $K_1<\frac{r_2}{\alpha_2}$.   By \eqref{CoefLk},
 the sign of $b_0(X)$ depends on a factor that is linear in $X$, for $X>0$. Furthermore,  $b_0(0)<0$ and $$b_0\left(\frac{r_2}{\alpha_2}\right)=\frac{K_2^2 r_1 (\alpha_2 K_1 - r_2) r_2 (1 + r_2)}{\alpha_2}<0.$$
It follows that $b_0(X)<0$ for all $X\in \left(0,\frac{r_2}{\alpha_2}\right)$.  
Since $b_2>0$, we have by \eqref{rk} that $r_{k_2}(X)<0$ and $r_{k_1}(X)>0$ for all $X\in \left(0, \frac{r_2}{\alpha_2}\right)$. Hence, the function $Y=r_{k_1}(X)$ determines $S_h$ uniquely  in the regions considered in {\it c)} and {\it d)}.

\begin{enumerate}
\item[{\it c)}]  Assume that $(X,Y)\in \mathcal{R}_4 \cup \mathcal{R}_{3_4}$.
For $X\in (0,X^*)$, $Y>Y^*$, so that by \eqref{signhkC3} and Lemma~\ref{usefulhk}{\it  b)}, $\mathcal{L}_k(X,k(X))>0$. 
By \eqref{LkX0}, $\mathcal{L}_k(X,0)=-k(F(X,0))$. Since for $0<X\leq K_1$, we have $0<F(X,0)\leq K_1$ and therefore also $k(F(X,0))>0$. For $K_1<X<\frac{r_2}{\alpha_2}$, the direction field indicates that  $F(X,0)<X<\frac{r_2}{\alpha_2}$, implying that $k(F(X,0))>0$. Hence,   for $X\in \left(0, \frac{r_2}{\alpha_2}\right)$, we have by \eqref{LkX0} that $\mathcal{L}_k(X,0)<0$. 
Since the next-iterate operator is continuous and $\mathcal{L}_k(X,k(X))>0$ but $\mathcal{L}_k(X,0)<0$, the sign of $\mathcal{L}_k(X,Y)$ must have changed sign below $Y=k(X)$. This means in turn that the root-curve $Y=r_{k_1}(X)$ must lie below the nullcline $Y=k(X)$ for $X\in(0,X^*)$.  Hence, $\mathcal{L}_k(X,Y)>0$, for all $(X,Y)\in \mathcal{R}_4 \cup \mathcal{R}_{3_4}$.

\item[{\it d)}] Assume that $(X,Y)\in \mathcal{R}_2 \cup\mathcal{R}_{1_2}$.
By Lemma~\ref{Lemnoint},  the only points where $Y=r_{k_1}(X)$ and the nullcline, $Y=k(X)$, intersect are the equilibrium points, $E_2$ and $E^*$.
For $X\in (X^*,K_1)$ and $0<Y<Y^*$, by \eqref{signhkC3} and  Lemma~\ref{usefulhk}{\it a)}, $\mathcal{L}_k(X,k(X))<0$ for $X\in(0,X^*)$. 
Since $Y=r_{k_1}(X)$ uniquely determines $S_k$ for $X\in \left(0,\frac{r_1}{\alpha_1}\right)$, and $\mathcal{L}_k(X,0)<0$ for $X\in \left(0,\frac{r_1}{\alpha_1}\right)$,  we know that $\mathcal{L}_k(X,Y)$  did not change sign below $Y=k(X)$.  Therefore,  $\mathcal{L}_k(X,Y)$     changes sign above the nullcline $Y=k(X)$ and so $\mathcal{L}_k(X,Y)<0$ for all $(X,Y)\in \mathcal{R}_2 \cup\mathcal{R}_{1_2}$.
\end{enumerate}
\vspace{-.85cm}
\end{proof}

\subsection{Proof of Lemma~\ref{box2}}\label{Pfbox2}
\begin{proof}
 The proof relies on the signs of the next-iterate operators  included in the augmented phase portrait shown in Fig.~\ref{Fig:C4general}a), based on  Lemma~\ref{LemsignCase4}. 

If there is a point $(X_t,Y_t) \in \mathcal{D}_1$ such that $(X_{t+1},Y_{t+1})\in \mathcal{D}_2$, then there must be a `++' region in $\mathcal{D}_1$,  since   the orbit would have to  jump across both nullclines in order to enter $\mathcal{D}_2$. Thus, it suffices to show that there is no `++' region in $\mathcal{D}_1$.
 In the proof of  Lemma~\ref{LemsignCase4}{\it a)} and {\it b)},  we proved that  $\mathcal{S}_h$ is determined by a unique positive root-curve $X=R_{h_1}(Y)$ for all $0<Y<\frac{r_1}{\alpha_1}$.    By \eqref{Rh},  $A_2>0$ and  since $A_0(0)=-K_1 K_2 r_1 <0$, from \eqref{Rh}, $R_{h_1}(0)>0$. Thus, $X=R_{h_1}(Y)$ intersects the $X$-axis at a value $X\in(0,K_1]$. This implies that every point  to the left of $X=R_{h_1}(Y)$ satisfies $\mathcal{L}_h(X,Y)<0$, since  there are  black `--' symbols in  region $\mathcal{R}_4\cup\mathcal{R}_{1_4}$, noting also that by Lemma~\ref{Lemnoint},  $X=R_{h_1}(Y)$ cannot intersect $Y=h(X)$ except at $E^*$ and/or $E_1$.  Thus, in order for a `++' region to exist in $\mathcal{D}_1$, the   nonnegative root-curve associated with $Y=k(X)$, namely   $Y=r_{k_1}(X)$  that uniquely determines $\mathcal{S}_k$  for $0<X<\frac{r_2}{\alpha_2}$ (see proof of  Lemma~\ref{LemsignCase4}{\it c)} and {\it d)})  would have to be to the right of (below)  $X=R_{h_1}(Y)$ in $\mathcal{D}_1$.  
 However,  $b_0(0)<0$, and since $b_2>0$,  it follows that $r_{k_1}(0)>0$, and so   $r_{k_1}(X)$ intersects the $Y$-axis at a  value $Y\in(0,K_2]$  due to the  gray `+' symbols in $\mathcal{R}_4\cup\mathcal{R}_{3_4} $.
    Thus,  $X=R_{h_1}(Y)$ is below $Y=r_{k_1}(X)$, at least for some $X\in (0,K_1)$. 
    Since, by Lemma~\ref{rootintersectD1}, the root-curves cannot intersect in $\mathcal{D}_1$, $Y=R_{h_1}(Y)$ must remain to the right of (below) the root-curve $Y=r_{k_1}(X)$. 
    Hence, no `++' region can exist in $\mathcal{D}_1$. 
\end{proof}

\bibliographystyle{abbrv}
\bibliography{version1}

\end{document}